\RequirePackage{fix-cm,amsmath}
\documentclass[smallextended]{svjour3}       
\smartqed  
\usepackage{amssymb}
\usepackage{amscd}
\usepackage{graphicx}
\usepackage{subcaption}
\usepackage{hyperref}
\usepackage{units}

\usepackage{enumitem}
\setlist[enumerate,1]{label=(\alph*), ref=(\alph*), leftmargin=*, align=left, labelwidth=!}
\usepackage{geometry}
\usepackage{authblk}
\usepackage{xargs}
\usepackage[linesnumbered]{algorithm2e}

\usepackage[most]{tcolorbox}
\usepackage{mathtools}

\tcbset{colback=yellow!10!white, colframe=black, 
        highlight math style= {enhanced, 
            colframe=red,colback=red!10!white,boxsep=0pt}
       }
 
 \DeclareMathOperator*{\argmin}{arg\,min}
 \newcommand\innprod[2]{\left\langle{}#1{},{}#2{}\right\rangle}%
 \newcommand\func[3]{#1:#2\rightarrow#3}
 \newcommand\R{\mathbb R}%
 \newcommand\N{\mathbb N}%
 \newcommand\C{\mathcal C}%
 \newcommand\K{\mathcal K}%
  \newcommand\W{\mathcal W}%
  \newcommand\X{\mathcal X}%
  \newcommand\Nf{\mathcal{N}^f_i}
  \newcommand\Nx{\mathcal{N}^x_i}
  \newcommand\Nnf{\mathcal{N}_i}
  \newcommand\MN{\mathcal{N}}
 
 \newcommand\eL{\mathcal{L}}%
 \newcommand\Lo{L_\Omega}%
 \newcommand\ov[1]{\overline #1}%
 \newcommand\interior{\mathrm{int}}%
 \newcommand\dist{\mathrm{dist}}%
 \newcommand\proj{\mathrm{proj}}%
 \newcommand\wh[1]{\widehat{#1}} %
 \newcommand\set[1]{\left\{{}#1{}\right\}} 
 \newcommandx\seq[3][{2=k\geq 0},{3={}}]{\{#1\}_{#2}^{#3}}
 \newcommandx\seqone[3][{2=k\geq 1},{3={}}]{\{#1\}_{#2}^{#3}}
 \newcommandx\subseq[3][{2=j\geq 0},{3={}}]{\{#1\}_{#2}^{#3}}
 
 \newtheorem{thm}{Theorem}
 \numberwithin{thm}{section}
 \newtheorem{lem}[thm]{Lemma}
 \newtheorem{prop}[thm]{Proposition}
 \newtheorem{cor}[thm]{Corollary}
 \newtheorem{rem}[thm]{Remark}
 
 \newtheorem{defin}[thm]{Definition}
 \newtheorem{fact}[thm]{Fact}
 \newtheorem{ass}{Assumption}

\journalname{...}
\begin{document}

\title{(Adaptive) Scaled gradient methods beyond locally H\"{o}lder smoothness: Lyapunov analysis, convergence rate and complexity 
}


\author{Susan Ghaderi        \and
        Morteza Rahimi\and
       Yves Moreau \and
      Masoud Ahookhosh   
}


\institute{S. Ghaderi, Y. Moreau\at
              Department of Electrical Engineering (ESAT), KU Leuven, Leuven, Belgium\\
              \email{susan.ghaderi@kuleuven.be,~Yves.Moreau@kuleuven.be}\\
           \and
            M. Rahimi, M. Ahookhosh \at
              Department of Mathematics, University of Antwerp, Middelheimlaan 1, B-2020 Antwerp, Belgium. \\
              Tel.: +123-45-678910\\
              \email{morteza.rahimi@uanwerp.be, masoud.ahookhosh@uantwerp.be}            \\
              The first author was partially supported by the Research Foundation Flanders (FWO) research project G081222N and UA BOF DocPRO4 project with ID 46929.
}

\date{Received: date / Accepted: date}

\maketitle

\begin{abstract}
This paper addresses the unconstrained minimization of smooth convex functions whose gradients are locally Hölder continuous. Building on these results, we analyze the Scaled Gradient Algorithm (SGA) under local smoothness assumptions, proving its global convergence and iteration complexity. Furthermore, under local strong convexity and the Kurdyka–{\L}ojasiewicz (KL) inequality, we establish linear convergence rates and provide explicit complexity bounds. In particular, we show that when the gradient is locally Lipschitz continuous, SGA attains linear convergence for any KL exponent. We then introduce and analyze an adaptive variant of SGA (AdaSGA), which automatically adjusts the scaling and step-size parameters. For this method, we show global convergence, and derive local linear rates under strong convexity.

\end{abstract}

\section{Introduction} \label{sec:introduction}
Let us consider the unconstrained minimization problem
\begin{equation}\label{eq:p}
        \displaystyle \min_{x\in\R^n} f(x),
\end{equation}
which satisfies the following assumption:
\begin{ass}\label{assumptionSG}
    We assume that
    \begin{enumerate}[label=(A\arabic*), ref=(A\arabic*)]
        \item\label{assumptionSG-1} The function $\func{f}{\R^n}{\R}$ is a convex smooth function which its gradient is locally H\"{o}lder;
        \item\label{assumptionSG-2} The set of minimizers $\mathcal{X}^*:=\argmin_{x\in \R^n} f(x)$ is nonempty and $f^*$ is the optimal value of \eqref{eq:p}.
    \end{enumerate}
\end{ass}

In the classical setting, it is usually assumed that the function $f$ has a global Lipschitz gradient (i.e., $f\in \C_{\mu, L}^{1,1}(\R^n)$); however, there are many interesting applications that their objective functions do not hold a Lipschitz continuous gradient; for examples, matrix factorization and nonnegative matrix factorization \cite{ahookhosh2023nonEuclidean,dragomir2021quartic}, orthogonal nonnegative matrix factorization \cite{ahookhosh2021multi}, symmetric nonnegative matrix tri-factorization \cite{ahookhosh2021block}, deep matrix factorization \cite{mukkamala2019bregman}, matrix completion \cite{ahookhosh2023nonEuclidean}, Poisson linear inverse problem \cite{bauschke2016descent}, $D$-optimal design \cite{lu2018relatively}, sparse phase retrieval \cite{latafat2022bregman}, neural networks \cite{mukkamala2019bregman}, and so on. While the smooth part of the objectives of such applications does not hold a global Lipschitz gradient, they, fortunately, have a  Lipschitz gradient on any compact subset of their domain.

The {\it gradient descent} method (also known as the steepest descent method) was devised by Louis Augustin Cauchy dating back to 1847, see \cite{cauchy1847methode}, which is one of the most popular and effective approaches for solving unconstrained optimization problems that takes an initial point $x^0$ and generates a sequence converging to a critical point of a function $f$ by moving along the anti-gradient direction with the step-size $\alpha_k>0$ (i.e., $x^{k+1}=x^k-\alpha_k\nabla f(x)$). Efficient implementation of this iterative method demands a proper technique for determining the step-size $\alpha_k$. For given $L>0$, if $f$ is {\it $L$-smooth} (smooth with Lipschitz continuous gradient and the modulus $L$, i.e., $f\in\C_L^{1,1}(\R^n)$), then one has a guarantee of convergence for constant step-size $\alpha_k\in (0,\tfrac{2}{L})$; see, e.g., \cite{nesterov2018lectures}. Additionally, if we assume the strong convexity of $f$ (i.e., $f\in\C_{\mu,L}^{1,1}(\R^n)$) and $\alpha_k\in (0,\tfrac{2}{\mu+L}]$, then it attains a linear rate; see, e.g., \cite[Theorem~2.1.15]{nesterov2018lectures}. One can use Polyak's optimal \cite{polyak1987introduction} or diminishing step-sizes; however, the first one requires the minimum $f(x^*)$, and the second usually results in slow convergence. Moreover, we can also specify $\alpha_k$ with either an exact or a monotone or nonmonotone inexact (e.g., Armijo, Goldstein, or Wolfe) line search; see, e.g., \cite{ahookhosh2012class,ahookhosh2017efficiency,amini2014inexact,bertsekas2016nonlinear,grippo1986nonmonotone,nocedal2006numerical}; however, applying line search imposes some computational cost to the method, which will be the bottleneck of the algorithm for large-scale problems. 

The {\it scaled gradient method} is an iterative descent method (i.e., $x^{k+1}=x^k+\alpha_kd^k$ with the descent direction $\innprod{\nabla f(x^k)}{d^k}<0$), where, compared to the generic gradient method, the anti-gradient direction $d^k=-\nabla f(x^k)$ is replaced by $d^k=-\gamma_k\nabla f(x^k)$ for the scaling parameter $\gamma_k>0$. It is well-known that the scaled gradient method can attain much faster convergence than the generic gradient method, where the parameter $\gamma_k$ has a key role in the efficiency of this method, and so various choices have been suggested in the literature. For example, spectral gradient directions, as a special case of scaled gradient direction, are given by
\begin{equation}\label{eq:sepectralGradDir}
    \gamma_k^{BB1}=\max\set{\gamma_{\min},\min\set{\alpha_k^{BB1},\gamma_{\max}}}, \quad \gamma_k^{BB2}=\max\set{\gamma_{\min},\min\set{\alpha_k^{BB2},\gamma_{\max}}}
\end{equation}
where $\gamma_k\in [\gamma_{\min},\gamma_{\max}]$ for $\gamma_{\max}\geq\gamma_{\min}>0$ and $\alpha_k^{BB1}$ and $\alpha_k^{BB1}$ are Barzilai-Borwein steps \cite{barzilai1988twopoint} given by
\begin{equation}\label{eq:BB1BB2}
    \alpha_k^{BB1}:=\frac{\|\nabla f(x^k)-\nabla f(x^{k-1})\|^2}{\innprod{\nabla f(x^k)-\nabla f(x^{k-1})}{x^k-x^{k-1}}}, \quad \alpha_k^{BB2}:=\frac{\innprod{\nabla f(x^k)-\nabla f(x^{k-1})}{x^k-x^{k-1}}}{\|x^k-x^{k-1}\|^2}.
\end{equation}
which were further studied in infinite-dimensional spaces in \cite{azmi2020analysis}. There are many variants of such steps for the spectral gradient direction; see, e.g., \cite{babaie2013modified,biglari2013scaling,dai2002modified}. It is clear that the scaled gradient methods are a subclass of the descent methods, where their global convergence has been studied under several line searches (cf. \cite{ahookhosh2017efficiency,biglari2013scaling,raydan1997barzilai}); however, they attain much better performance if they are combined with nonmonotone line searches; see, e.g., \cite{ahookhosh2017efficiency,birgin2000nonmonotone,birgin2014spectral,raydan1997barzilai}.

Recently, an {\it adaptive gradient method} was proposed in \cite{malitsky2019adaptive} for smooth convex optimization problems, requiring neither the global Lipschitz property of $f$ nor any line search techniques for establishing global convergence. In detail, this algorithm is given by
\begin{equation}
    x^{k+1}=x^k-\alpha_k\nabla f(x^k), \quad \alpha_k:=\left\{
    \begin{array}{ll}
        \infty &~~ \mathrm{if} ~k=0, \vspace{1mm}\\
        \min\set{\alpha_{k-1} \sqrt{1+\frac{\alpha_{k-1}}{\alpha_{k-2}}},\frac{1}{2L_k}} &~~ \mathrm{if}~ k\geq 1,
    \end{array} 
    \right.
\end{equation}
where $L_k$ stands as an approximation of the local Lipschitz modulus given by
\[
    L_k:=\frac{\|\nabla f(x^k)-\nabla f(x^{k-1})\|}{\|x^k-x^{k-1}\|},
\]
Let us emphasize that this innovative convergence analysis overlooks the monotonicity paradigm of the descent methods, as is the case for descent methods with nonmonotone line searches \cite{ahookhosh2012class,ahookhosh2017efficiency,amini2014inexact,grippo1986nonmonotone}; however, it constitutes the monotonicity of a specific Lyapunov function which is a necessary condition to verify its convergence analysis of the developed approach. More recently, this idea has been further extended in \cite{malitsky2023adaptive} for finding a less conservative step-size, leading to a faster convergence for the gradient method. The favorable properties of the above-mentioned adaptive step-size and fast convergence of the scaled gradient directions motivate the quest to study the scaled gradient methods under local H\"{o}lder smoothness umbrella, which is much weaker than the global $L$-smooth assumption. 

\subsection{Contribution} \label{sec:contribution}
Our contributions are threefold:
\begin{itemize}
    \item[\textbf{(i)}] \textbf{Analysis of the Scaled Gradient Algorithm (SGA).}
        we develop a comprehensive framework for studying smooth convex functions whose gradients are only locally H\"{o}lder continuous. We establish several fundamental inequalities and characterizations that extend classical smoothness results to the local H\"{o}lder setting.
        Building on this framework, we analyze the SGA under local smoothness assumptions. Global convergence and complexity bounds are derived, and linear convergence is established when the objective function satisfies either local strong convexity or the KL inequality. In particular, it is shown that when the gradient is locally Lipschitz continuous, SGA achieves linear convergence for any KL exponent.
    \item[\textbf{(ii)}] \textbf{Adaptive Scaled Gradient Algorithm (AdaSGA).}
        The second part of the paper introduces AdaSGA, which dynamically adjusts the scaling and step-size parameters during the iterations. A Lyapunov-based analysis is carried out to demonstrate the monotone descent property and global convergence of the algorithm. Moreover, under local strong convexity, we establish linear convergence and provide complexity guarantees consistent with those of SGA.
\end{itemize}

\subsection{Organization} \label{sec:organization}
This paper has five sections, besides this introductory section. Section~\ref{sec:prelimNota} provides the necessary notation, some definitions, and facts. In Section~\ref{sec:SGA}, we study the convergence of the scaled gradient method and establish the linear convergence under extra local strong convexity and the KL inequality separately. In Section~\ref{sec:AdaSGA}, we introduce our adaptive scaled gradient method, study its convergence and complexity, and establish its linear rate under extra local strong convexity. Finally, Section~\ref{sec:conclusion} delivers some concluding remarks.

\section{Preliminaries and notation} \label{sec:prelimNota}

\subsection{\textbf{Notation}}
In this paper, $\R^n$ is $n$-dimensional real Euclidean space equipped with the standard inner product \(\innprod{\cdot}{\cdot}\) and the Euclidean norm \(\|\cdot\|=\sqrt{\innprod{\cdot}{\cdot}}\). The notation $\N$ stands for the set of natural numbers and $\N_{0}:=\N\cup \{0\}$.
The open and closed balls of radius $r>0$ centered at $x\in\R^n$ are respectively expressed as $B(x;r)$ and $\overline B(x;r)$. The interior, closure, and boundary of a set $S\subseteq\R^n$ are respectively denoted as $\interior S, \overline S$, and $\mathrm{bdry} S:=\overline S\setminus \interior S$. The Euclidean distance of a point $x\in\R^n$ to a nonempty set $S\subseteq\R^n$ is given by $\dist(x;S)=\inf_{z\in S} \|z-x\|$.
Moreover, the Euclidean projection of the point $x$ onto $S$ is given by $\proj_{S}(x):=\argmin_{z\in S} \|z-x\|$.
A function $\func{h}{\R^n}{\R}$ is coercive if $\lim_{\|x\|\to\infty}h(x)= \infty$.

\begin{fact}[\textbf{Convergence of a sequence with positive elements}]
\label{fac:convRate1} \cite[Lemma 15]{boct2020proximal}
    Let $\theta,\beta>0$ and let $\seq{s^k}$ be a sequence in $\R_+$ converging to $0$ and satisfying
    \begin{equation}\label{eq:sk}
        (s^k)^{\theta}\leq\beta(s^k-s^{k+1}),
    \end{equation}
    holds for $k\in\N$ sufficiently large.
    Then, the following assertions hold:
    \begin{enumerate}
        \item\label{fac:convRate1-1} If $\theta=0$, the sequence $\seq{s^k}$ converges to $0$ in a finite number of steps;
        \item\label{fac:convRate1-2} If $\theta\in(0,1]$, the sequence $\seq{s^k}$ converges linearly to $0$ with the rate $1-\tfrac{1}{\beta}$;
        \item\label{fac:convRate1-3} If $\theta>1$, there exists $\sigma>0$ such that 
        \[
            0\leq s^k\leq \sigma k^{-\tfrac{1}{\theta-1}},
        \]
        for $k\in\N_{0}$ sufficiently large.
    \end{enumerate}
\end{fact}

The following fact provides some sufficient conditions for the existence of an optimal solution of problem \eqref{eq:p}, along with some properties and first-order characterization of optimal solutions.

\begin{fact}\label{fac:optimality}\cite{rockafellar2011variational}
    Let $\func{f}{\R^n}{\R}$ be a lower semicontinuous function. Then, the following assertions hold:
    \begin{enumerate}
        \item
            If the function $f$ is coercive or has a bounded level set, then $\X^*$ is nonempty and compact;
        \item
            If the function $f$ satisfies Assumption~\ref{assumptionSG}, then, the optimal set $\X^{*}$ is a nonempty, closed, and convex set. Furthermore, $x^{*} \in \X^*$ if and only if $\nabla f(x^*)=0$.
    \end{enumerate}
\end{fact}
\section{Gradient methods under local smoothness} \label{sec:SGA}
This section deals with two main topics: (i) adaptation of the classical results of smoothness and strong convexity to their local counterparts; and (ii) description and analysis of the scaled gradient algorithm. We begin by verifying and adapting the known smoothness and strong convexity results that are necessary in analyzing scaled gradient methods and in the remainder of our study.

\subsection{\textbf{$\nu$-H\"{o}lder continuity, $L$-smoothness, $\mu$-strong convexity, and KL inequality}} \label{sec:localSmoothStrConv}

\subsubsection{\textbf{$\nu$-H\"{o}lder continuity and $L$-smoothness}} \label{sec:HOSMO}

Let $D \subseteq \mathbb{R}^n$ be a nonempty set, and let $\func{F}{D}{\mathbb{R}^m}$ be a mapping. 
The mapping $F$ is said to be \emph{$\nu$-H\"{o}lder continuous} on a nonempty set $S \subseteq D$ with $\nu \in (0,1]$ 
if there exists a constant $L > 0$, called the \emph{H\"{o}lder constant}, such that
\[
    \|F(x) - F(y)\| \leq L \|x - y\|^{\nu}, \quad \forall\, x, y \in S.
\]
The smallest such constant is referred to as the \emph{H\"{o}lder modulus} of $F$ on $S$ and is given by
\begin{equation}\label{eq:LipschitzModulus}
    L_{S} := \sup_{\substack{x,y \in S \\ x \neq y}} 
    \frac{\|F(x) - F(y)\|}{\|x - y\|^{\nu}}.
\end{equation}
The mapping $F$ is called \emph{$\nu$-H\"{o}lder continuous} if it satisfies the above condition on the entire set $D$.

For a given point $\bar{x} \in D$, the mapping $F$ is said to be \emph{locally $\nu$-H\"{o}lder} at $\bar{x}$ if there exists a radius $\delta > 0$ such that 
$F$ is $\nu$-H\"{o}lder continuous on $D \cap B(\bar{x}; \delta)$. 
More generally, $F$ is \emph{locally $\nu$-H\"{o}lder on $S$} if it is locally $\nu$-H\"{o}lder at every point $x \in S$. 
Finally, $F$ is simply referred to as \emph{locally $\nu$-H\"{o}lder} (or simply \emph{locally H\"{o}lder} when $\nu$ is clear from the context) if this property holds on the entire set $D$.

Let $S \subseteq \mathbb{R}^n$ be a nonempty set. 
We denote by $\mathcal{C}^{1}(S)$ the class of functions $\func{h}{\mathbb{R}^n}{\mathbb{R}}$ that are continuously differentiable on $S$. 
A function $\func{h}{\mathbb{R}^n}{\mathbb{R}}$ is said to be \emph{$(\nu, L)$-H\"{o}lder smooth} (or simply \emph{$L$-smooth} when $\nu = 1$) on $S$ 
if it is differentiable on $S$ and its gradient $\nabla h$ is $\nu$-H\"{o}lder continuous with constant $L$ on $S$. 
The class of such functions is denoted by $\mathcal{C}^{1,\nu}_{L}(S)$. 
Moreover, the class of functions $\func{h}{\mathbb{R}^n}{\mathbb{R}}$ that are differentiable on $S$ and its gradient $\nabla h$ is locally $\nu$-H\"{o}lder on $S$
is denoted by $\mathcal{C}^{1,\nu+}(S)$ (or simply $\mathcal{C}^{1,+}(S)$ when $\nu = 1$; see, e.g.,~\cite[Chapter~9]{rockafellar2011variational}).
In the special case $\nu = 1$, the terms \lq\lq$\nu$-H\"{o}lder" and \lq\lq H\"{o}lder" reduce to \lq\lq Lipschitz". 

Let $c>0$ and $S\subseteq \R^n$ be a nonempty set, and let $\func{h}{\R^n}{\R}$ be a differentiable function on $S$. The function $h$ is said to have $\tfrac{1}{c}$-cocoercive gradient on $S$ if
\begin{equation}\label{eq:cocoercivity}
    \innprod{\nabla h(y)-\nabla h(x)}{y-x}\geq \tfrac{1}{c} \|\nabla h(y)-\nabla h(x)\|^2, \quad \forall x, y\in S.
\end{equation}

In the following, we establish several characterizations of $L$-smoothness and $\tfrac{1}{L}$-cocoerciveness for continuously differentiable functions defined on a convex set with nonempty interior. These characterizations are presented in \cite[Theorem 2.1.5]{nesterov2018lectures} for the entire space $\mathbb{R}^n$, and \cite[Lemma 3.2]{perez2021enhanced} extends them to convex open sets. In Proposition~\ref{pro:SmoothChar}, we further generalize these results to an arbitrary convex set $S$ with nonempty interior.

\begin{prop}[\textbf{$L$-smoothness and $\tfrac{1}{L}$-cocoerciveness characterization}]\label{pro:SmoothChar}
    Let $L>0$, let $D\subseteq \R^n$ be a nonempty set with $\interior D\neq \emptyset$, and let $\func{h}{D}{\R}$ be a convex smooth function on a nonempty convex set $S\subseteq \interior D$ with $\interior S\neq \emptyset$. Then, the following assertions are equivalent:
    \begin{enumerate}
        \item\label{pro:SmoothChar-1} The function $h$ is a $L$-smooth function on $S$, i.e., $h\in \C^{1,1}_L(S)$;
        \item\label{pro:SmoothChar-2} The function $\tfrac{L}{2}\|x\|^2-h(x)$ is convex on $S$;
        \item\label{pro:SmoothChar-3} The function $h$ has $\tfrac{1}{L}$-cocoercive gradient on $S$;
        \item\label{pro:SmoothChar-4} $0\le\innprod{\nabla h(y)-\nabla h(x)}{y-x}\leq L \|y-x\|^2,$ for any $x, y\in S$;
        \item\label{pro:SmoothChar-5} $0\leq h(y)- h(x)-\innprod{\nabla h(x)}{y-x} \leq \tfrac{L}{2}\|y-x\|^2$, for any $x,y \in S$.
    \end{enumerate}
\end{prop}
\begin{proof}
    It is evident that Assertions \ref{pro:SmoothChar-2}, \ref{pro:SmoothChar-4}, and \ref{pro:SmoothChar-5} are equivalent. Additionally, if $S$ is an open set, Assertions \ref{pro:SmoothChar-1}, \ref{pro:SmoothChar-2}, and \ref{pro:SmoothChar-3} are equivalent by virtue of \cite[Lemma 3.2]{perez2021enhanced}. To complete the proof it suffices to show that if either assertion holds on $\interior S$, it is valid on $S$ as well.\\
    Let the function $h$ be a $L$-smooth function on $\interior S$.
    Let $\ov x,\ov y\in S$ and $\ov x\neq \ov y$. Based on \cite[Theorem 6.1]{rockafellar1970convex}, since $\interior S\neq \emptyset$ and $S$ is convex, there exist $\seq{x^k}\subseteq \interior S$ and $\seq{y^k}\subseteq \interior S$ converging to $\ov x$ and $\ov y$, respectively. For all $k\in \N$ sufficiently large, we have $x^k\neq y^k$ and
    \[\tfrac{\|\nabla h(x^k)-\nabla h(y^k)\|}{\|x^k-y^k\|}\leq L.\]
    Inasmuch as $h$ is smooth on $S$, taking limit as $k\to\infty$ yields
    \[\tfrac{\|\nabla h(\ov x)-\nabla h(\ov y)\|}{\|\ov x-\ov y\|}\leq L.\]
    The same argument but with different inequalities can be used for Assertions \ref{pro:SmoothChar-3} and \ref{pro:SmoothChar-4}.\qed
\end{proof}

We next extend our analysis to the class of $(\nu,L)$-H"older smooth functions with $\nu \in (0,1]$ and $L>0$. For any $x,y\in \R^n$, define
$$d_{x,y}:=\|\nabla h(y) - \nabla h(x)\|^{\tfrac{1-\nu}{\nu}}(\nabla h(y) - \nabla h(x)).$$
According to \cite[Theorem 2.1.5]{nesterov2018lectures}, a function $h:\R^n\to\R$ is $L$-smooth on the whole space $\mathbb{R}^n$ if and only if 
\begin{equation*}
        h(y)\geq h(x)+\innprod{\nabla h(x)}{y-x}+\tfrac{1}{2L}\|\nabla h(y)-\nabla h(x)\|^2,~~\forall x,y\in \R^n.
\end{equation*}
However, this equivalence may fail for functions defined only on open subsets of $\mathbb{R}^n$, as illustrated by a counterexample in \cite{drori2018properties}. To overcome this limitation, several studies have established analogous characterizations for functions defined on more general domains; see, for instance, \cite{drori2018properties,wachsmuth2022simple}.
Proposition~\ref{pro:HSmoothChar}~\ref{pro:HSmoothChar-3} further extends this characterization to functions defined on arbitrary convex sets under mild regularity assumptions.


\begin{prop}[\textbf{$(\nu,L)$-H\"{o}lder smoothness characterization}]\label{pro:HSmoothChar}
    Let $L>0$ and $\nu \in (0,1)$, let $D\subseteq \R^n$ be a nonempty set with $\interior D\neq \emptyset$, and let $\func{h}{D}{\R}$ be a convex smooth function on a nonempty convex set $S\subseteq \interior D$. If the function $h$ is $(\nu,L)$-H\"{o}lder smooth on $S$, then the following assertions hold:
    \begin{enumerate}
        \item\label{pro:HSmoothChar-1}
            For any $x,y \in S$, with constant $\C=L$,
            \begin{equation}\label{eq:HSmoothChar-1}
                0\le \innprod{\nabla h(y)-\nabla h(x)}{y-x}\leq \C \|y-x\|^{1+\nu};
            \end{equation}
        \item\label{pro:HSmoothChar-2}
            For any $x,y \in S$, with constant $\C=\tfrac{L}{1+\nu}$,
            \begin{equation}\label{eq:HSmoothChar-2}
                0\leq h(y)- h(x)-\innprod{\nabla h(x)}{y-x} \leq \C\|y-x\|^{1+\nu};
            \end{equation}
        \item\label{pro:HSmoothChar-3}
            For any $x,y \in S$ such that $y - \tfrac{d_{x,y}}{L^{\nicefrac{1}{\nu}}}\in S$, with constant $\C=\tfrac{\nu}{(1+\nu)L^{\nicefrac{1}{\nu}}}$,
            \begin{equation}\label{eq:HSmoothChar-3}
                h(y)\geq h(x)+\innprod{\nabla h(x)}{y-x}+\C\|\nabla h(y)-\nabla h(x)\|^{\tfrac{1+\nu}{\nu}};
            \end{equation}  
        \item\label{pro:HSmoothChar-4}
            For any $x,y \in S$ such that $y - \tfrac{d_{x,y}}{L^{\nicefrac{1}{\nu}}}\in S$ and $x + \tfrac{d_{x,y}}{L^{\nicefrac{1}{\nu}}}\in S$, with constant $\C=\tfrac{2\nu}{(1+\nu)L^{\nicefrac{1}{\nu}}}$,
            \begin{equation}\label{eq:HSmoothChar-4}
                \innprod{\nabla h(y)-\nabla h(x)}{y-x}\geq \C \|\nabla h(y)-\nabla h(x)\|^{\tfrac{1+\nu}{\nu}}.
            \end{equation}
    \end{enumerate}
\end{prop}
\begin{proof}
    \ref{pro:HSmoothChar-1} This result is a direct consequence of $(\nu,L)$-H\"{o}lder smoothness and convexity of $h$ together with Cauchy-Schwarz inequality.\\
    \ref{pro:HSmoothChar-2} Let $x,y \in S$. From \ref{pro:HSmoothChar-1} and convexity of $h$, we deduce
    {\small
    \begin{align*}
        0\le h(y)- h(x)-\innprod{\nabla h(x)}{y-x} = \int_0^1 \innprod{\nabla h(x+t(y-x))-\nabla h(x)}{y-x} dt
        \le \int_0^1 Lt^{\nu}\|x-y\|^{1+\nu}dt = \tfrac{L}{1+\nu}\|x-y\|^{1+\nu}.
    \end{align*}
    }
    \ref{pro:HSmoothChar-3}
    Let us consider $x,y\in S$ such that $y - \tfrac{d_{x,y}}{L^{\nicefrac{1}{\nu}}}\in S$. Define $g(z):=h(z)-\innprod{\nabla h(x)}{z}$, which is a $\C^{1,\nu}_L(S)$ convex smooth function and attaints its minimum on $S$ at $z=x$. 
    By virtue of Assertion \ref{pro:HSmoothChar-2},
    $$g(z)\le g(w)+\innprod{\nabla g(w)}{z-w} + \tfrac{L}{1+\nu}\|z-w\|^{1+\nu}, \quad\quad \forall z, w \in S.$$
    Using this inequality, we come to
    \begin{align*}
        h(x) - \innprod{\nabla h(x)}{x} &= g(x) = \min_{z\in S} g(z) \le \min_{z\in S} \left\{g(y)+\innprod{\nabla g(y)}{z-y} + \tfrac{L}{1+\nu}\|z-y\|^{1+\nu}\right\}\\
        &= g(y)-\tfrac{\nu}{(1+\nu)L^{\nicefrac{1}{\nu}}}\|\nabla g(y)\|^{\tfrac{1+\nu}{\nu}}
        =h(y)-\innprod{\nabla h(x)}{y} - \tfrac{\nu}{(1+\nu)L^{\nicefrac{1}{\nu}}}\|\nabla h(y)-\nabla h(x)\|^{\tfrac{1+\nu}{\nu}},
    \end{align*}
    where the third equality follows from the fact that the minimum of the quadratic model is attained at $$z = \overline{y} - \tfrac{1}{L^{\nicefrac{1}{\nu}}}\|\nabla g(y)\|^{\tfrac{1-\nu}{\nu}}\nabla g(y)=y - \tfrac{d_{x,y}}{L^{\nicefrac{1}{\nu}}}\in S.$$
    Assertion \ref{pro:HSmoothChar-4} is derived from \ref{pro:HSmoothChar-3} by summing inequality~\eqref{eq:HSmoothChar-3} evaluated at $(x,y)$ and $(y,x)$. This completes the proof.\qed
\end{proof}

\begin{cor}\label{cor:HSmoothChar}
    Let $L>0$ and $\nu \in (0,1]$, let $D\subseteq \R^n$ be a nonempty set with $\interior D\neq \emptyset$, and let $\func{h}{D}{\R}$ be a convex smooth function on a nonempty convex set $S\subseteq \interior D$. Then, the following statements hold:
    \begin{enumerate}
        \item\label{cor:HSmoothChar-1}
            If \eqref{eq:HSmoothChar-1} holds on $S$ with constant $\ov\C$, then \eqref{eq:HSmoothChar-2} holds on $S$ with constant $\tfrac{\ov\C}{1+\nu}$.
        \item\label{cor:HSmoothChar-2}
            If \eqref{eq:HSmoothChar-2} holds  on $S$ with constant $\ov\C$, then for any $x,y\in S$ such that $y - \tfrac{d_{x,y}}{L^{\nicefrac{1}{\nu}}}\in S$, the inequality  \eqref{eq:HSmoothChar-3} holds with constant $\tfrac{\nu}{(1+\nu)^{\nicefrac{(1+\nu)}{\nu}} {\ov\C}^{\nicefrac{1}{\nu}}}$. 
        \item\label{cor:HSmoothChar-3}    
            If \eqref{eq:HSmoothChar-2} (resp. \eqref{eq:HSmoothChar-3}) holds on $S$ with constant $\ov \C$, then the inequality \eqref{eq:HSmoothChar-1} (resp. \eqref{eq:HSmoothChar-4}) holds on $S$ with constant $2\ov\C$.
        \item\label{cor:HSmoothChar-4}
            If \eqref{eq:HSmoothChar-4} holds on $S$ with constant $\ov\C$, then
            $h$ is $(\nu,\tfrac{1}{\C^\nu})$-H\"{o}lder smooth on $S$.
        \item\label{cor:HSmoothChar-5}
            Let $\nu=1$. If $h$ is a $L$-smooth function on $S$, i.e., $h\in \C^{1,1}_L(S)$, then for any $x,y\in S$ such that $y-\tfrac{1}{L}\big(\nabla h(y)-\nabla h(x)\big)\in S$ one has
            \begin{equation}\label{eq:smoothconvex2}
                h(y)\geq h(x)+\innprod{\nabla h(x)}{y-x}+\tfrac{1}{2L}\|\nabla h(y)-\nabla h(x)\|^2.
            \end{equation}
            Conversely, if the inequality \eqref{eq:smoothconvex2} holds for any $x,y\in S$, then $h$ is a $L$-smooth function on $S$.
    \end{enumerate}
\end{cor}
\begin{proof}
    The claim is immediate.\qed
\end{proof}

\subsubsection{\textbf{$\mu$-Strong convexity}} \label{sec:STC}

Let $\mu>0$ and $S\subseteq \R^n$ be a nonempty convex set, and let $h\in \C^1(S)$. The function $h$ is said to be {\it $\mu$-strongly convex} on $S$ if
\begin{equation}\label{eq:strongconvex}
    h(y)\geq h(x)+\innprod{\nabla h(x)}{y-x}+\tfrac{\mu}{2}\|y-x\|^2, \quad \forall x,y\in S.
\end{equation}

The next fact provides several characterizations for class of strongly convex functions.

\begin{fact}[\textbf{$\mu$-strong convexity characterization}]\label{fac:StrongChar}\cite{nesterov2018lectures}
    Let $\mu>0$ and $S\subseteq \R^n$ be a nonempty convex set, and let $h\in \C^{1}(S)$. The following statements are equivalent:
    \begin{enumerate}
        \item\label{fac:StrongChar-1} $h$ is a $\mu$-strongly convex function on $S$;
        \item\label{fac:StrongChar-2} the function $h(x)-\tfrac{\mu}{2}\|x\|^2$ is convex on $S$;
        \item\label{fac:StrongChar-3} $\innprod{\nabla h(y)-\nabla h(x)}{y-x}\geq \mu\|y-x\|^2$, for any $x,y\in S$.
    \end{enumerate}
\end{fact}

The largest constant $\mu$ satisfying the inequality \eqref{eq:strongconvex} (indicating $\mu$-strong convexity on $S$), denoted by $\mu_S$, is given by
\begin{equation}\label{eq:LipschitzModulus}
    \mu_{S}:=\inf_{x,y\in S,~x\neq y} \tfrac{\innprod{\nabla h(y)-\nabla h(x)}{y-x}}{\|y-x\|^2}.
\end{equation}

Let $h\in \C^1(\R^n)$. The function $h$ is said to be {\it locally strongly convex} if for any $x\in \R^n$, there exists some scalars $\delta_x>0$ and $\mu_x>0$, such that $h$ is $\mu_x$-strongly convex on $B(x;\delta_x)$.

The following proposition indicates that the locally strongly convex functions are strongly convex on any convex and compact set.

\begin{thm}\label{thm:localglobal-strong}
    Let $h\in \C^1(\R^n)$ be locally strongly convex and let $S\subseteq\R^n$ be a nonempty convex and compact set. Then, there exists $\mu> 0$ such that $h$ is $\mu$-strongly convex on $S$.
\end{thm}
\begin{proof}
    To begin, we establish the existence of a fixed constant $\mu> 0$ such that, for each $x\in S$, the following inequality holds:
    \begin{equation}\label{eq:localstrongconvex}
        h(y)\geq h(z)+\innprod{\nabla h(z)}{y-z}+\tfrac{\mu}{2}\|y-z\|^2, \quad \forall z,y\in B(x;\delta_x),
    \end{equation}
    for some $\delta_x>0$. The local strong convexity of $h$ implies that for each point $x\in S$, there exist constants
    $\mu_{x}\geq 0$ and $\delta_{x}>0$ such that the inequality \eqref{eq:localstrongconvex} holds with $\mu_x$ replacing $\mu$ for all
    $z,y \in B(x;\delta_{x})$ . 
    The collection of open balls, $\{B(x;\delta_{x})\}_{x\in S}$ covers $S$, and by the compactness principle a finite subcollection of them covers $S$ as well.
    Thus, there exist points $x_i\in S$ and radii $\delta_i>0$, $i=1,2,\ldots,m,$ such that 
    $S\subseteq \bigcup_{i=1}^m B(x_i;\delta_{i})$.
    Define $\overline{\mu}:=\max\{\mu_i : ~~i=1,2,\ldots,m\}$.
    Then, for each $x\in S$, there exists some $\delta_x>0$ such that $B(x;\delta_x) \subseteq B(x_i;\delta_{i})$
    and \eqref{eq:localstrongconvex} holds on $\mathbb{B}(x;\delta_x)$ with the fixed constant $\overline{\mu}$. Thus, by virtue of Fact \ref{fac:StrongChar}~\ref{fac:StrongChar-2}, the function $h(x)-\tfrac{\overline{\mu}}{2}\|x\|^2$ is locally convex on the convex and compact set $S$, i.e., it is convex on $S$ by \cite[Corollary 2.4]{li2010some}. Therefore, $h$ is $\overline{\mu}$-strongly convex on $S$, due to Fact \ref{fac:StrongChar}~\ref{fac:StrongChar-1}, confirming the result.\qed
\end{proof}

It is worth noting that every $\C^1(\R^n)$ strongly convex function has a unique optimal solution. However, this property does not hold for functions that are only locally strongly convex. For instance, the function $x \mapsto e^x$ is locally strongly convex but does not possess a minimizer. The following proposition establishes that, for a locally strongly convex function, if an optimal solution exists, it must be unique.

\begin{prop}\label{pro:SingleStrong}
    Let $h\in \C^1(\R^n)$ be locally strongly convex function with $\X^*=\{y\in \R^n : h(y)=\inf_{x\in \R^n} h(x)\}$. Then, either $\X^*=\emptyset$ or $\X^*=\{x^*\}$ is singleton, i.e., $h$ has a unique optimal solution.
\end{prop}
\begin{proof}
    Let us assume that $\X^*\neq \emptyset$.
    Suppose, for the sake of contradiction, that there exist two distinct minimizers $\overline{x}, \hat{x} \in \X^*$ with $\overline{x} \neq \hat{x}$. Define the compact and convex set $S:=\overline{B}(0;\delta)$ where $\delta:=\max\{\|\ov{x}\|, \|\hat{x}\|\}$. By virtue of Theorem \ref{thm:localglobal-strong}, there exists $\mu>0$ such that $h$ is $\mu$-strongly convex on $S$. Hence, it holds that
    $$h(\ov{x})-h(\hat{x})\ge \innprod{\nabla h(\hat{x})}{\ov{x}-\hat{x}}+\tfrac{\mu}{2}\|\ov{x}-\hat{x}\|^2=\tfrac{\mu}{2}\|\ov{x}-\hat{x}\|^2>0,$$
    making a contradiction. Therefore, the optimal solution must be unique.\qed
\end{proof}

Now, we consider functions satisfying both $L_S$-smoothnes and $\mu_S$-strong convexity on a nonempty convex set $S\subseteq\R^n$ with $\interior S\neq \emptyset$ and show an essential inequality in the next result, which is the extension of \cite[Theorem~2.1.12]{nesterov2018lectures}.

\begin{lem}\label{lem:LocalStrLipCont}
    Let $S\subseteq \R^n$ be a nonempty convex set with $\interior S\neq \emptyset$, and let $\func{h}{\R^n}{\R}$ with $h\in \C^1(S)$ be $L_S$-smoothness and $\mu_S$-strongly convex on $S$. Then, the following assertions hold:
    \begin{enumerate}
        \item\label{lem:LocalStrLipCont-1}
            One has $\mu_S\leq L_S$ and
            \begin{equation}\label{eq:lowBoundStrConvLipCon}
                \innprod{\nabla h(y)-\nabla h(x)}{y-x}\geq \tfrac{\mu_S L_S}{\mu_S +L_S} \|y-x\|^2 + \tfrac{1}{\mu_S +L_S} \|\nabla h(y)-\nabla h(x)\|^2, \quad \forall x, y\in S;
            \end{equation}
        \item\label{lem:LocalStrLipCont-2}
            If $\mu_S=L_S$, then
            $$\innprod{\nabla h(y)-\nabla h(x)}{y-x}=L_S\|y-x\|^2,\quad \quad \|\nabla h(y)-\nabla h(x)\|=L_S\|y-x\|,\quad\forall x,y\in S.$$
            Furthermore, there exist $a\in \R^n$ and $b\in \R$ such that $h(x)=\tfrac{L}{2}\|x\|^2+\innprod{a}{x}+b$ for any $x\in S$, i.e., $h$ is a quadratic function on $S$. If $0\in S$, then $a=\nabla h(0)$ and $b=h(0)$.
    \end{enumerate}
\end{lem}
\begin{proof}
    \ref{lem:LocalStrLipCont-1} Let us define the function $\func{\varphi}{S}{\R}$ given by 
    \begin{align*}
        \varphi(x)=h(x)-\tfrac{\mu_S}{2} \|x\|^2.
    \end{align*}
    It follows from Proposition~\ref{pro:SmoothChar}~\ref{pro:SmoothChar-4} together with Fact \ref{fac:StrongChar}~\ref{fac:StrongChar-3} that
    \begin{align}\label{eq:lowandupBound}
        \mu_S \|y-x\|^2\leq \innprod{\nabla h(y)-\nabla h(x)}{y-x}\leq L_S \|y-x\|^2, \quad \forall x, y\in S,
    \end{align}
    i.e.,
    \begin{align*}
        0\leq \innprod{\nabla \varphi(y)-\nabla \varphi(x)}{y-x}\leq (L_S-\mu_S) \|y-x\|^2, \quad \forall x, y\in S.
    \end{align*}
    This implies that $\mu_S \le L_S$ and that $\varphi \in \mathcal{C}_{L_S-\mu_S}^{1,1}(S)$ by Proposition~\ref{pro:SmoothChar}~\ref{pro:SmoothChar-1}. If $\mu_S = L_S$, the result is immediate. Otherwise, Proposition~\ref{pro:SmoothChar} guarantees that
    \begin{align*}
        \innprod{\nabla \varphi(y)-\nabla \varphi(x)}{y-x}\geq \tfrac{1}{L_S-\mu_S} \|\nabla \varphi(y)-\nabla \varphi(x)\|^2, \quad \forall x, y\in S,
    \end{align*}
    which coincides precisely with \eqref{eq:lowBoundStrConvLipCon}.
    
    \ref{lem:LocalStrLipCont-2} Let us define $L:=\mu_S=L_S$. It follows from \eqref{eq:lowandupBound}, Cauchy-Schwarz inequality, and $L_S$-smoothness of $f$ that
    $$L_S \|y-x\|^2=\innprod{\nabla h(y)-\nabla h(x)}{y-x}= \|\nabla h(y)-\nabla h(x)\|\|y-x\| \leq L_S\|y-x\|^2, \quad \forall x, y\in S,$$
    giving us desired equalities. On the other hand, the function $h(\cdot) - \tfrac{L}{2}\|\cdot\|^2$ is both convex and concave on $S$, according to Proposition \ref{pro:SmoothChar}~\ref{pro:SmoothChar-2} and Fact \ref{fac:StrongChar}~\ref{fac:StrongChar-2}. Thus, $h(\cdot) - \tfrac{L}{2}\|\cdot\|^2$ is a linear function on $S$, i.e., there exists $a\in \R^n$ and $b\in \R$ such that
    $$h(x)-\tfrac{L}{2}\|x\|^2=\innprod{a}{x}+b,\quad\forall x\in S.$$
    It is evident that if $0\in S$, $h(0) = b$, and $\nabla h(0) = a$.
    \qed
\end{proof}

The following corollary, as a direct consequence of Lemma \ref{lem:LocalStrLipCont}, extends the results from global $L$-smoothness and strong convexity to their local counterparts.

\begin{cor}\label{cor:LocalStrLipCont}
    Let $S\subseteq \R^n$ be a nonempty convex and compact set with $\interior S\neq \emptyset$, and let $\func{h}{\R^n}{\R}$ be a $h\in \C^1(\R^n)$ locally strongly convex function with locally Lipschitz gradient. Then, the following assertions hold:
    \begin{enumerate}
        \item\label{cor:LocalStrLipCont-1}
            One has $\mu_S\leq L_S$ and
            \begin{equation*}
                \innprod{\nabla h(y)-\nabla h(x)}{y-x}\geq \tfrac{\mu_S L_S}{\mu_S +L_S} \|y-x\|^2 + \tfrac{1}{\mu_S +L_S} \|\nabla h(y)-\nabla h(x)\|^2, \quad \forall x, y\in S;
            \end{equation*}
        \item\label{cor:LocalStrLipCont-2}
            If $\mu_S=L_S$, then
            $$\innprod{\nabla h(y)-\nabla h(x)}{y-x}=L_S\|y-x\|^2,\quad \quad \|\nabla h(y)-\nabla h(x)\|=L_S\|y-x\|,\quad\forall x,y\in S.$$
            Furthermore, there exist $a\in \R^n$ and $b\in \R$ such that $h(x)=\tfrac{L}{2}\|x\|^2+\innprod{a}{x}+b$ for any $x\in S$, i.e., $h$ is a quadratic function on $S$. If $0\in S$, then $a=\nabla h(0)$ and $b=h(0)$.
    \end{enumerate}
\end{cor}

\subsubsection{\textbf{Kurdyka-{\L}ojasiewicz inequality}} \label{sec:LojaIneq}

In this section, we introduce Kurdyka-{\L}ojasiewicz inequality which plays a key role in studying the convergence rate and improving the rate obtained for the sequence generated by our gradient methods.

\begin{defin}[\textbf{Kurdyka-{\L}ojasiewicz inequality}] \label{def:LojaIneqDef}
    Let $\func{h}{\R^n}{\R}$ be a continuously differentiable and convex function with $h^*=\inf_{x\in \R^n} h(x)$ and let the set of minimizers $\X^*=\set{x\in \R^n~|~ h(x)=h^*}$ be nonempty.
    The function $h$ is said to satisfy the Kurdyka-{\L}ojasiewicz (KL) inequality if for any optimal point $x^*\in \X^*$, there exist constants $\rho, \delta>0$ and $\vartheta\in (0,1)$ such that
    \begin{equation}\label{eq:KLIneq}
        (h(x)-h^*)^\vartheta\leq \rho \|\nabla h(x)\|, \quad \forall x\in B(x^*;\delta).
    \end{equation}
    Additionally, the function $h$ is said to satisfy the KL inequality on a nonempty set $S\subseteq \R^n$ with constants $\vartheta_\Omega\in (0,1)$ and $\rho_\Omega>0$, if the inequality \eqref{eq:KLIneq} with $\vartheta=\vartheta_S$ and $\rho=\rho_S$ holds for all $x\in S$.
\end{defin}

Let us emphasize that the inequalities \eqref{eq:KLIneq} is satisfied for a broad class of functions appearing in practice. For instance, in 1963, Stanislaw {\L}ojasiewicz proved that these inequalities hold for analytic functions \cite{lojasiewicz1963propriete,lojasiewicz1993geometrie}. In addition, in 1998, Kurdyka showed that the KL inequality \eqref{eq:KLIneq} is satisfied for differentiable functions definable in an $o$-minimal structure; see \cite{kurdyka1998gradients}.

The following theorem establishes that the KL inequality \eqref{eq:KLIneq} holds uniformly on any compact set with suitable KL parameters $\vartheta$ and $\rho$ provided that the convex function satisfies the KL gradient inequality.

\begin{thm}\label{thm:KL}
   Let $\func{h}{\R^n}{\R}$ be a continuously differentiable and convex function. If the function $h$ satisfies KL inequality, then the following assertions hold:
   \begin{enumerate}
       \item\label{thm:KL-1}
            For any $\ov x\in \R^n$, there exist constants $\rho, \delta>0$ and $\vartheta\in (0,1)$ such that
            \[
            (h(x)-h^*)^\vartheta\leq \rho \|\nabla h(x)\|, \quad \forall x\in B(\ov x;\delta).\]
       \item\label{thm:KL-2}
            For any compact set $S\subseteq \R^n$, the function $h$ satisfies the KL inequality on $S$.
   \end{enumerate}
\end{thm}
\begin{proof}
    \ref{thm:KL-1} Let $\ov x\in \R^n$. If $\|\nabla h(\ov x)\|=0$, then $\ov x\in \X^*$, and the claim follows trivially. Assume now that $\|\nabla h(\ov x)\|>0$. Then, there exists a constant $\delta>0$ such that $\|\nabla h(x)\|>0$ for all $x\in \ov B(\ov x;\delta)$. Then, for any $\vartheta \in (0,1)$ we obtain
    \[(h(x)-h^*)^\vartheta\leq \max_{z\in \ov B(\ov x,\delta)} (h(z)-h^*)^\vartheta \le \Big(\max_{z\in \ov B(\ov x,\delta)} (h(z)-h^*)^\vartheta  \max_{z\in \ov B(\ov x,\delta)} \|\nabla h(z)\|^{-1}\Big) \|\nabla h(x)\|, \quad \forall x\in B(\ov x,\delta).\]
    \ref{thm:KL-2} By Assertion \ref{thm:KL-1}, for any $\ov x\in S$, there exist constants $\rho_{\ov x}, \delta_{\ov x}>0$ and $\vartheta_{\ov x} \in (0,1)$ such that
    \((h(x)-h^*)^{\vartheta_{\ov x}}\leq \rho_{\ov x} \|\nabla h(x)\|\) for any \(x\in B(\ov x;\delta_{\ov x}).\)
    Since $S$ is a compact set, there exist $m\in \N$ and points $x_1, x_2, \ldots, x_m\in S$ with corresponding radii $\delta_1,\delta_2,\ldots, \delta_m>0$ such that $S\subseteq \bigcup_{i=1}^m B(x_i;\delta_i)$. Let $\rho_i > 0$ and $\vartheta_i \in (0,1)$ be the KL parameters corresponding to $x_i$, $i = 1, \dots, m$. Define
    $\vartheta_S:=\max_{i}\vartheta_i$ and $\rho_S:= \max_{i}\kappa_i\rho_i$ where $\kappa_i := max_{z\in \ov B(x_i;\delta_i)} (h(z)-h^*)^{\vartheta_S-\vartheta_i}$. Let us consider $x\in S$ arbitrary. Then, there exists an index $i$ such that $x\in \ov B(x_i;\delta_i)$ and
    \[(h(x)-h^*)^{\vartheta_S} = (h(x)-h^*)^{\vartheta_S - \vartheta_i} (h(x)-h^*)^{\vartheta_i} \le \kappa_i \rho_i \|\nabla h(x)\|\le \rho_S \|\nabla h(x)\|,\]
    which establishes the claim.\qed
\end{proof}

\subsection{{\bf SGA (Scaled gradient algorithm)}}
In this section, we investigate the scaled gradient descent method. Let us begin with the following assumption, which plays a key role in obtaining convergence results. We recall that, throughout the paper, we consider problem~\eqref{eq:p} satisfying Assumption~\ref{assumptionSG}.

\begin{ass}\label{ass:LevelBound}
    The lower level set $\Lambda(x^0):=\left\{x \in \mathbb{R}^{m} \mid f(x) \leq f(x^{0})\right\}$ is bounded with $x^{0}\in \R^n$.
\end{ass}

\begin{rem}
    Let Assumption \ref{ass:LevelBound} hold.
    Let us define the compact and convex set $\Omega:=\overline{B}(0;\ov r)$ where
    $$\ov r:=(1+2(1+\nu)^{\nicefrac{1}{\nu}})r\quad\quad\text{and}\quad\quad r:=\max\|x\| ~~~\text{s.t.}~~x\in \Lambda(x^0).$$
    It is evident that $\X^*\subseteq\Lambda(x^0)\subseteq \overline{B}(0;r)\subseteq \Omega$. On account of Assumptions \ref{assumptionSG}~\ref{assumptionSG-1} and \ref{ass:LevelBound}, $\Lambda(x^0)$ is a compact set, ensuring the well-definedness of $r\ge 0$. Furthermore, since $\nabla f$ is locally $\nu$-H\"{o}lder, it is automatically $\nu$-H\"{o}lder continuous on the compact set $\Omega$ with constant $\Lo$, i.e., the function $f$ is $(\nu,\Lo)$-H\"{o}lder smooth on $\Omega$.
    \qed
\end{rem}

Next, we analyze the global convergence, convergence rate, and complexity of the gradient descent method.

\RestyleAlgo{boxruled}
\begin{algorithm}[ht!]
\DontPrintSemicolon
\KwIn{$x^0\in\R^ n$,~$\Lo>0$,~ $0<\nu\le 1$,~ $k=0$;}
\Begin{ 
    \While{
        the stopping criteria does not hold 
    }{
    Select a step-size $\alpha_k\in \big(0,\big(\tfrac{1+\nu}{\Lo}\big)^{\nicefrac{1}{\nu}}\big]$;\;
    Set $x^{k+1}=x^k-\alpha_k \|\nabla f(x^k)\|^{\tfrac{1-\nu}{\nu}}\nabla f(x^k)$ and $k=k+1$;
    }
    $x^{\mathrm{best}}=x^k$;\;
}
\KwOut{$x^{\mathrm{best}}$} 
\caption{SGA: Scaled Gradient Algorithm \label{alg:SGA}}
\end{algorithm}

It is worth highlighting three key distinctions between Algorithm~\ref{alg:SGA} and the standard gradient descent method, which typically assumes a globally Lipschitz continuous gradient:

\begin{enumerate}
    \item[1.]
    In Algorithm~\ref{alg:SGA}, the function $f$ is assumed to have a locally $\nu$-H\"{o}lder gradient, rather than a globally Lipschitz continuous one. This generalization allows the algorithm to handle a broader class of convex functions.
    \item[2.]
    When the $\nu$-H\"{o}lder constant \(\Lo\) is known, one can establish the global convergence of the iterates to an optimal solution by selecting an appropriate step-size within a prescribed interval. Moreover, linear convergence follows under additional conditions such as the KL inequality or local strong convexity. In contrast, when \(\Lo\) is unknown, choosing a step-size that diminishes to zero, such as diminishing, non-summable, square-summable but non-summable, or geometrically decaying step-sizes, ensures that it eventually falls within the desired interval, thereby guaranteeing (linear) convergence. Notice that in implementation, the algorithm is run for a limited iterations making a lower bound for step-sizes.
    \item[3.]
    If the function is \((\nu,L)\)-H\"{o}lder smooth, then \(\Lo\le L\). Consequently, a slightly larger admissible interval for the step-size can be considered, which may improve the algorithm’s performance. Moreover, if the function is \(\Lo\)-smooth on $\Omega$, then it is \((\nu,\tilde{L})\)-H\"{o}lder smooth on $\Omega$ for any $\nu\in (0,1]$ with a constant \(L_\nu\ge \Lo\). This observation provides additional flexibility in choosing the step-size and potentially enhances the convergence behavior.
\end{enumerate}

\subsubsection{\textbf{Global Convergence}}

\begin{thm}[\textbf{Sufficient decrease}]
\label{thm:SufDecSGA}
    Let Assumption \ref{ass:LevelBound} hold. If the sequence $\seq{x^k}$ is generated by Algorithm \ref{alg:SGA}, the following assertions hold:
    \begin{enumerate}
        \item\label{thm:SufDecSGA-1} $\seq{x^k}\subseteq \Lambda(x^0)$,
        \begin{equation}\label{eq:squarefk1fk-1}
        f(x^{k+1}) \leq f(x^k) - \left(\alpha_k-\tfrac{\Lo}{1+\nu}\alpha_k^{1+\nu}\right) \|\nabla f(x^k)\|^{\tfrac{1+\nu}{\nu}},\quad \forall k\in\N_0,
        \end{equation}
        and $f(x^{k+1})\leq f(x^k)$ for any $k\in \N_0$;
        \item\label{thm:SufDecSGA-2} Let $\alpha_k\in \big(0,\tfrac{4\nu}{(1+\nu)\Lo^{\nicefrac{1}{\nu}}}\big]$. For any $x^*\in \X^*$, one has
        \begin{equation}\label{eq:squareXk1Xk-1}
        \|x^{k+1}-x^*\|^2 \leq \|x^k-x^*\|^2-\left(\tfrac{4\nu\alpha_k}{(1+\nu){\Lo}^{\nicefrac{1}{\nu}}}-\alpha_k^2\right) \|\nabla f(x^k)\|^{\tfrac{2}{\nu}},\quad \forall k\in\N_0,
        \end{equation}
        and $\|x^{k+1}-x^*\| \leq \|x^k-x^*\|$ for any $k\in \N_0$;
        \item\label{thm:SufDecSGA-3} In Assertions \ref{thm:SufDecSGA-1} and \ref{thm:SufDecSGA-2}, the maximum decrease of the inequalities \eqref{eq:squarefk1fk-1} and \eqref{eq:squareXk1Xk-1} are respectively given for $\alpha_k=\tfrac{1}{\Lo^{\nicefrac{1}{\nu}}}$ and $\alpha_k=\tfrac{2\nu}{(1+\nu)\Lo^{\nicefrac{1}{\nu}}}$.
    \end{enumerate}
\end{thm}
\begin{proof}
    \ref{thm:SufDecSGA-1}  We proceed by induction to verify $\seq{x^k}\subseteq \Lambda(x^0)$. For $k=0$, we note that $x^0\in \Lambda(x^0)$. Assume that $x^k\in \Lambda(x^0)$ for some $k\in \N_0$.
    First, we show $x^k\in \Omega$.
    Based on $(\nu,\Lo)$-H\"{o}lder smoothness of $f$ on $\Omega$, we deduce
    $$\|\nabla f(x^k)\| = \|\nabla f(x^k)-\nabla f(x^*)\| \leq \Lo \|x^k-x^*\|^{\nu} \leq \Lo (\|x^k\| +\|x^*\|)^{\nu} \leq (2r)^{\nu}\Lo,$$
    where $x^*\in \X^*$ is arbitrary. Applying this inequality and $\alpha_k \le \big(\tfrac{1+\nu}{\Lo}\big)^{\nicefrac{1}{\nu}}$, we obtain
    $$\|x^{k+1}\| = \Big\|x^k-\alpha_k\|\nabla f(x^k)\|^{\tfrac{1-\nu}{\nu}}\nabla f(x^k)\Big\| \leq \|x^k\| + \alpha_k\|\nabla f(x^k)\|^{\tfrac{1-\nu}{\nu}}\|\nabla f(x^k)\| < r + \big(\tfrac{1+\nu}{\Lo}\big)^{\nicefrac{1}{\nu}} 2r\Lo^{\nicefrac{1}{\nu}} = \ov r,$$
    i.e., $x^{k+1}\in \Omega$.
    Now, By Proposition \ref{pro:HSmoothChar}~\ref{pro:HSmoothChar-2}, we observe that
        \begin{align*}
            f(x^{k+1})&\leq f(x^{k}) + \innprod{\nabla f(x^k)}{x^{k+1}-x^k} + \tfrac{\Lo}{1+\nu}\|x^{k+1}-x^k\|^{1+\nu}\\
            &= f(x^{k}) 
            -\alpha_k \|\nabla f(x^k)\|^{\tfrac{1+\nu}{\nu}} + \tfrac{\Lo}{1+\nu}\alpha_k^{1+\nu} \|\nabla f(x^k)\|^{\tfrac{1+\nu}{\nu}}\\
            &= f(x^k) - \left(\alpha_k-\tfrac{\Lo}{1+\nu}\alpha_k^{1+\nu}\right) \|\nabla f(x^k)\|^{\tfrac{1+\nu}{\nu}}
            \leq f(x^k)\leq f(x^0),
        \end{align*}
    ensuring $x^{k+1}\in \Lambda(x^0)$, inequality \eqref{eq:squarefk1fk-1}, and $f(x^{k+1})\le f(x^{k})$.\\
    \ref{thm:SufDecSGA-2} Since $\alpha_k \leq \tfrac{4\nu}{(1+\nu)\Lo^{\nicefrac{1}{\nu}}}\le \big(\tfrac{1+\nu}{\Lo}\big)^{\nicefrac{1}{\nu}}$, $\seq{x^k}\subseteq \Lambda(x^0)$ by Assertion \ref{thm:SufDecSGA-1}. Let $x^*\in\X^*$. It holds that
    $$\Big\|x^k - \tfrac{1}{\Lo^{\nicefrac{1}{\nu}}}\|\nabla f(x^k)\|^{\tfrac{1-\nu}{\nu}}\nabla f(x^k)\Big\| \le \|x^k\| + \tfrac{1}{\Lo^{\nicefrac{1}{\nu}}} 2r\Lo^{\nicefrac{1}{\nu}} \le 3r< \ov r,$$
    implying $x^k - \tfrac{1}{\Lo^{\nicefrac{1}{\nu}}}\|\nabla f(x^k)\|^{\tfrac{1-\nu}{\nu}}\nabla f(x^k)\in \Omega$. Similarly, one can show that $x^* + \tfrac{1}{\Lo^{\nicefrac{1}{\nu}}}\|\nabla f(x^k)\|^{\tfrac{1-\nu}{\nu}}\nabla f(x^k)\in \Omega$. 
    Thus, it follows from Proposition \ref{pro:HSmoothChar}~\ref{pro:HSmoothChar-4} that
    \[
        \innprod{\nabla f(x^k)}{x^k-x^*} =\innprod{\nabla f(x^k)-\nabla f(x^*)}{x^k-x^*}\geq \tfrac{2\nu}{(1+\nu){\Lo}^{\nicefrac{1}{\nu}}} \|\nabla f(x^k)-\nabla f(x^*)\|^{\tfrac{1+\nu}{\nu}}=\tfrac{2\nu}{(1+\nu){\Lo}^{\nicefrac{1}{\nu}}} \|\nabla f(x^k)\|^{\tfrac{1+\nu}{\nu}}.
    \]
    Using this inequality, we deduce
    \begin{align*}
        \|x^{k+1}-x^*\|^2 &= \Big\|x^k-x^* -\alpha_k\|\nabla f(x^k)\|^{\tfrac{1-\nu}{\nu}}\nabla f(x^k)\Big\|^2\\
        &= \|x^k-x^*\|^2-2\alpha_k\|\nabla f(x^k)\|^{\tfrac{1-\nu}{\nu}}\innprod{\nabla f(x^k)}{x^k-x^*}+\alpha_k^2\|\nabla f(x^k)\|^{\tfrac{2}{\nu}}\\
        &\le \|x^k-x^*\|^2- \tfrac{4\nu\alpha_k}{(1+\nu){\Lo}^{\nicefrac{1}{\nu}}} \|\nabla f(x^k)\|^{\tfrac{2}{\nu}} +\alpha_k^2\|\nabla f(x^k)\|^{\tfrac{2}{\nu}}\\
        &= \|x^k-x^*\|^2-\left(\tfrac{4\nu\alpha_k}{(1+\nu){\Lo}^{\nicefrac{1}{\nu}}}-\alpha_k^2\right) \|\nabla f(x^k)\|^{\tfrac{2}{\nu}}
         \le \|x^k-x^*\|^2.
    \end{align*}
    \ref{thm:SufDecSGA-3} The functions $\alpha\mapsto\alpha-\tfrac{\Lo}{1+\nu}\alpha^{1+\nu}$ and $\alpha\mapsto\tfrac{4\nu\alpha}{(1+\nu){\Lo}^{\nicefrac{1}{\nu}}}-\alpha^2$ respectively are the concave functions on $\big(0,\big(\tfrac{1+\nu}{\Lo}\big)^{\nicefrac{1}{\nu}}\big]$ and $\big(0,\tfrac{4\nu}{(1+\nu)\Lo^{\nicefrac{1}{\nu}}}\big]$, and their maximum respectively attain at $\alpha_k=\tfrac{1}{\Lo^{\nicefrac{1}{\nu}}}$ and $\alpha_k=\tfrac{2\nu}{(1+\nu)\Lo^{\nicefrac{1}{\nu}}}$. This completes the proof.\qed
\end{proof}

Theorem \ref{thm:GlobconvSGA} presents that the sequence $\seq{x^k}$ globally converges to an optimal solution.

\begin{thm}[\textbf{Global convergence of SGA}]
\label{thm:GlobconvSGA}
    Let Assumption \ref{ass:LevelBound} hold. If the sequence $\seq{x^k}$ is generated by Algorithm \ref{alg:SGA} with dynamic or constant step-size $\alpha_k \in [\ov\alpha , \tfrac{1}{\Lo^{\nicefrac{1}{\nu}}}]$ at which $0<\ov \alpha<\tfrac{1}{\Lo^{\nicefrac{1}{\nu}}}$, the following assertions hold:
    \begin{enumerate}
         \item\label{thm:GlobconvSGA-1}
            One has
            \begin{equation}\label{eq:squarefk1fk-2}
                f(x^{k+1}) \leq f(x^k)- \tfrac{\Lo\ov\alpha^{1+\nu} }{1+\nu}\|\nabla f(x^k)\|^{\tfrac{1+\nu}{\nu}},\quad \forall k\in\N_0;
            \end{equation}
        \item\label{thm:GlobconvSGA-2}
            We have $\sum_{k\ge 0} \|\nabla f(x^k)\|^{\tfrac{1+\nu}{\nu}} <\infty$ and $\displaystyle\lim_{k\to\infty} \|\nabla f(x^k)\|=0$;
        \item\label{thm:GlobconvSGA-3}
        The limiting points of $\seq{x^k}$ are optimal solutions and $\displaystyle\lim_{k\to\infty} f(x^k)=f^*$;
        \item\label{thm:GlobconvSGA-4}
            If $\alpha_k \in [\ov\alpha , \tfrac{2\nu}{(1+\nu)\Lo^{\nicefrac{1}{\nu}}}]$ where $0<\ov \alpha<\tfrac{2\nu}{(1+\nu)\Lo^{\nicefrac{1}{\nu}}}$, then for any $x^*\in \X^*$, one has
            \begin{equation}\label{eq:squareXk1Xk-4}
                \|x^{k+1}-x^*\|^2 \leq \|x^k-x^*\|^2 -\ov\alpha^2\|\nabla f(x^k)\|^{\tfrac{2}{\nu}} \quad \forall k\in\N_0.
            \end{equation}
            Moreover, the sequence $\seq{x^k}$ converges to an optimal solution.
    \end{enumerate}
\end{thm}
\begin{proof}
    \ref{thm:GlobconvSGA-1}  The function $\alpha\mapsto\alpha-\tfrac{\Lo}{1+\nu}\alpha^{1+\nu}$ is a concave function on $[\ov\alpha , \tfrac{1}{\Lo^{\nicefrac{1}{\nu}}}]$ attaining its minimum at $\alpha=\ov\alpha$ or $\alpha=\tfrac{1}{\Lo^{\nicefrac{1}{\nu}}}$. Thus, 
    \begin{equation}\label{eq:LowerBoundCoeNab}
        \tfrac{\Lo}{1+\nu}\ov\alpha^{1+\nu}\le\min\left\{\ov\alpha-\tfrac{\Lo}{1+\nu}\ov\alpha^{1+\nu}, \tfrac{\nu}{(1+\nu)\Lo^{\nicefrac{1}{\nu}}}\right\}\le \alpha_k-\tfrac{\Lo}{1+\nu}\alpha_k^{1+\nu},\quad \forall k\in\N_0.
    \end{equation}
    Using this inequality, \eqref{eq:squarefk1fk-1} yields \eqref{eq:squarefk1fk-2}.\\
    \ref{thm:GlobconvSGA-2}, summing up both sides of \eqref{eq:squarefk1fk-2} from $k=0$ to $k=N$ gives us
    \begin{align}\label{eq:squarefk1fk-3}
        \tfrac{\Lo\ov\alpha^{1+\nu} }{1+\nu} \sum_{k=0}^{N}\|\nabla f(x^k)\|^{\tfrac{1+\nu}{\nu}} \le
        \sum_{k=0}^{N} \big( f(x^k) - f(x^{k+1})\big) = f(x^0) - f(x^{N+1}) \le f(x^0) - f^* <\infty,
    \end{align}
    ensuring $\sum_{k\ge 0} \|\nabla f(x^k)\|^{\tfrac{1+\nu}{\nu}} <\infty$ and $\displaystyle\lim_{k\to\infty} \|\nabla f(x^k)\|=0$.\\ \ref{thm:GlobconvSGA-3}
    Due to Theorem \ref{thm:SufDecSGA}~\ref{thm:SufDecSGA-1} and Assumption \ref{ass:LevelBound}, $\seq{x^k}\subseteq \Lambda(x^0)$ is a bounded sequence. Assume that $\subseq{x^{k_j}}$ is a subsequence of $\seq{x^k}$ converging to \(\ov x\in \Lambda(x^0)\). Then, by Assertion \ref{thm:GlobconvSGA-2},
    $\|\nabla f(\ov x)\|=0,$
    i.e., $\ov x\in \X^*$. Furthermore, we get $\displaystyle\lim_{k\to\infty} f(x^k)=f(\ov x)=f^*$.\\
    \ref{thm:GlobconvSGA-4} Let $\alpha_k \in [\ov\alpha , \tfrac{2\nu}{(1+\nu)\Lo^{\nicefrac{1}{\nu}}}]$ where $0<\ov \alpha<\tfrac{2\nu}{(1+\nu)\Lo^{\nicefrac{1}{\nu}}}$. Similar to Assertion \ref{thm:GlobconvSGA-1}, one can show that
    $$\ov\alpha^{2}\le\min\left\{\tfrac{4\nu}{(1+\nu){\Lo}^{\nicefrac{1}{\nu}}}\ov\alpha-\ov\alpha^2, \tfrac{4\nu^2}{(1+\nu)^{2}\Lo^{\nicefrac{2}{\nu}}}\right\}\le \tfrac{4\nu}{(1+\nu){\Lo}^{\nicefrac{1}{\nu}}}\alpha_k-\alpha_k^2,\quad \forall k\in\N_0.$$
    which together with \eqref{eq:squareXk1Xk-1} imply \eqref{eq:squareXk1Xk-4}. Let now $\ov x \in \X^*$ is a limiting point of $\seq{x^k}$. Then, the sequence $\seq{\|x^{k}-\ov x\|}$ is decreasing and bounded from below, which confirms its convergence, i.e.,
    $$\lim_{k\to\infty}\|x^{k}-\ov x\|=\lim_{j\to\infty}\|x^{k_j}-\ov x\|  = 0,$$
    completing the proof.\qed
\end{proof}

In the next theorem, we employ Theorem \ref{thm:GlobconvSGA} to provide the global complexity of the SGA to guarantee $\|\nabla f(x^k)\|\leq \varepsilon$, for a specified accuracy
parameter $\varepsilon>0$.
Let us reserve the notation $\Nnf(\varepsilon)$ as the initial iteration at which the inequality $\|\nabla f(x^k)\|\leq \varepsilon$ is fulfilled.

\begin{thm}[\textbf{Convergence rate and complexity analysis of SGA}]\label{thm:ConvRat-Com-SGA}
     Let Assumption \ref{ass:LevelBound} hold. If $\seq{x^k}$ is generated by Algorithm \ref{alg:SGA} with dynamic or constant step-size $\alpha_k \in [\ov\alpha , \tfrac{1}{\Lo^{\nicefrac{1}{\nu}}}]$ at which $0<\ov \alpha<\tfrac{1}{\Lo^{\nicefrac{1}{\nu}}}$, one has
     \begin{equation}\label{eq:UppBounabla1}
     \min_{0\le k\le N} \|\nabla f(x^k)\| \le \Big(\tfrac{(1+\nu)\big(f(x^0)-f^*\big)}{\Lo\ov\alpha^{1+\nu}N}\Big)^{\tfrac{\nu}{1+\nu}}, \quad\quad \forall N\in\N.
     \end{equation}
     Moreover, for a given accuracy $\varepsilon>0$, the number of iterations to guarantee $\|\nabla f(x_{k})\| \leq \varepsilon$, denoted by $\Nnf(\varepsilon)$, satisfies
    \begin{equation}\label{eq:UppBouIt}
        \Nnf(\varepsilon)\leq K_{0}:=\left\lceil\tfrac{(1+\nu)\big(f(x^0)-f^*\big)}{\Lo\ov\alpha^{1+\nu}\varepsilon^{\tfrac{1+\nu}{\nu}}} +1\right\rceil.
    \end{equation}
\end{thm}
\begin{proof}
    The inequality \eqref{eq:UppBounabla1} follows from \eqref{eq:squarefk1fk-3}.  Additionally, it holds that
    $$\min_{0\le k\le K_{0}} \|\nabla f(x^k)\| \le \Big(\tfrac{(1+\nu)\big(f(x^0)-f^*\big)}{\Lo\ov\alpha^{1+\nu}K_{0}}\Big)^{\tfrac{\nu}{1+\nu}} \le \varepsilon,$$
    verifying $\Nnf(\varepsilon)\leq K_{0}$, adjusting the claim.\qed
\end{proof} 

\subsubsection{\textbf{Linear convergence under local strong convexity}}

Theorem \ref{thm:linConvSGA-LSC} establishes linear convergence of the iterates distance $\seq{\|x^k - x^*\|}$, the function value gap $\seq{f(x^k)-f^*}$, and the gradient norm $\seq{\|\nabla f(x^k)\|}$,  under local Lipschitzness of $\nabla f$, ($\nu=1$), and local strong convexity of $f$. Specifically, in this setting, the function $f$ is $\mu_\Omega$-strongly convex on $\Omega$.

\begin{thm}[\textbf{Linear convergence rate under local strong convexity}]
\label{thm:linConvSGA-LSC}
    Let Assumption \ref{ass:LevelBound} hold, let $\nabla f$ be locally Lipschitz $(\nu=1)$, and let the function $f$ be locally strongly convex where $\mu_\Omega<\Lo$. If the sequence $\seq{x^k}$ is generated by Algorithm~\ref{alg:SGA} with dynamic or constant step-size $\alpha_k\in [\ov\alpha,\tfrac{1}{\Lo}]$ at which $0<\ov\alpha<\tfrac{1}{\Lo}$, then the following assertions hold:
     \begin{enumerate}
        \item\label{thm:linConvSGA-LSC-1}
            The sequence $\seq{\|x^{k}-x^*\|}$ and $\seq{f(x)-f^*}$ converge Q-linearly to $0$ with the rates $\sqrt{1-q_0}$ and \(1-q_1\), respectively, i.e.,
            \begin{equation}\label{eq:uperlinearstrong2}
                \begin{split}
                    \|x^{k+1}-x^*\|^2 \leq \left(1-q_0\right) \|x^{k}-x^*\|^2,\quad\quad\forall k\in \N_0,
                \end{split}
            \end{equation}
            \begin{equation}\label{eq:LineStrFun}
                \begin{split}
                    f(x^{k+1})-f^* \leq \left(1-q_1\right) \big(f(x^{k})-f^*\big),\quad\quad\forall k\in \N_0,
                \end{split}
            \end{equation}
            where $x^*$ is an optimal solution, $q_0:=\tfrac{2\ov{\alpha}\mu_\Omega \Lo}{\mu_\Omega+\Lo}$, and $q_1:=\tfrac{\mu_\Omega^2\ov{\alpha}^2}{4}$;
        \item\label{thm:linConvSGA-LSC-2}
            The sequence $\seq{\|\nabla f(x^k)\|}$ converges R-linearly to $0$, i.e.,
            \[
            \|\nabla f(x^k)\| \le \tfrac{\sqrt{2}\Lo}{\sqrt{\mu_\Omega}} (f(x^0)-f^*)^{1/2}  \left(1-q_1\right)^{k/2}, \quad\quad \forall k\in \N_0.
            \]
     \end{enumerate}
\end{thm}
\begin{proof}
    \ref{thm:linConvSGA-LSC-1} By Proposition \ref{pro:SingleStrong}, the solution set $\X^*=\{x^*\}$ is singleton. Moreover, Theorem \ref{thm:localglobal-strong} ensures that the function $f$ is $\mu_{\Omega}$-strongly convex on $\Omega\supseteq \Lambda(x^0)$.
    Applying Lemma~\ref{lem:LocalStrLipCont} with $x=x^*$ and $y=x^k$ yields
    \begin{align*}
        \innprod{\nabla f(x^k)}{x^k-x^*}\geq \tfrac{\mu_\Omega \Lo}{\mu_\Omega+\Lo} \|x^k-x^*\|^2 + \tfrac{1}{\mu_\Omega+\Lo} \|\nabla f(x^k)\|^2.
    \end{align*}
    Using this inequality together with $\ov \alpha \le \alpha_k \le \tfrac{1}{\Lo}$, we obtain
    \begin{align*}
        \|x^{k+1}-x^*\|^2 &= \|x^k-x^* -\alpha_k \nabla f(x^k)\|^2= \|x^k-x^*\|^2-2\alpha_k\innprod{\nabla f(x^k)}{x^k-x^*}+\alpha_k^2\|\nabla f(x^k)\|^2\\
        &\leq \left(1-\tfrac{2\mu_\Omega \Lo\alpha_k}{\mu_\Omega+\Lo}\right)\|x^k-x^*\|^2+\alpha_k\left(\alpha_k-\tfrac{2}{\mu_\Omega+\Lo}\right)\|\nabla f(x^k)\|^2\\
        &\leq \left(1-\tfrac{2\mu_\Omega \Lo\ov\alpha}{\mu_\Omega+\Lo}\right)\|x^k-x^*\|^2,
    \end{align*}
    verifying \eqref{eq:uperlinearstrong2}.
    Let us prove \eqref{eq:LineStrFun}.
    It follows from $\mu_\Omega$-strong convexity of $f$ that 
    \[
    \tfrac{\mu_\Omega}{2} \|x^k -x^*\|^2 \le f(x^k) -f^* +\tfrac{\mu_\Omega}{2} \|x^k -x^*\|^2 \le \innprod{\nabla f(x^k)}{x^k-x^*} \le \|\nabla f(x^k)\| \|x^k-x^*\|,
    \]
    implying \(\tfrac{\mu_\Omega}{2} \|x^k -x^*\| \le \|\nabla f(x^k)\|\). Using this inequality and $\Lo$-smoothness of $f$, we come to
    \[
    f(x^k)-f^* \le \innprod{\nabla f(x^*)}{x^k-x^*} + \tfrac{\Lo}{2} \|x^k -x^*\|^2 \le \tfrac{2\Lo}{\mu_\Omega^2} \|\nabla f(x^k)\|^2.
    \]
    Now, by Proposition \ref{pro:HSmoothChar}~\ref{pro:HSmoothChar-2} and applying former inequality together with \eqref{eq:LowerBoundCoeNab} with $\nu=1$, it holds that
    \begin{align*}
        f(x^{k+1})-f^*&\leq f(x^{k}) -f^* + \innprod{\nabla f(x^k)}{x^{k+1}-x^k} + \tfrac{\Lo}{2}\|x^{k+1}-x^k\|^{2}\\
        &\le f(x^{k}) -f^* -\alpha_k \|\nabla f(x^k)\|^{2} + \tfrac{\Lo}{2}\alpha_k^{2} \|\nabla f(x^k)\|^{2}\\ &\le f(x^{k})-f^* - \left(\alpha_k-\tfrac{\Lo}{2}\alpha_k^2\right) \|\nabla f(x^k)\|^2\\
        &\le f(x^k)-f^* - \tfrac{\mu_\Omega^2\ov{\alpha}^2}{4}  \big (f(x^k) -f^*\big),
    \end{align*}
    ensuring \eqref{eq:LineStrFun}.\\
    \ref{thm:linConvSGA-LSC-2} It follows from $\Lo$-smoothness and $\mu_\Omega$-strong convexity of $f$ on $\Omega$ that
    \[
    \|\nabla f(x^k)\|^2 \le \Lo^2 \|x^k-x^*\|^2 \le \tfrac{2\Lo^2}{\mu_\Omega} (f(x^k)-f^*), \quad\quad \forall k\in \N_0.
    \]
    Therefore, the claim is justified using inequality \eqref{eq:LineStrFun}, completing the proof.\qed
\end{proof}

\begin{rem}\label{rem:L=mu}
    Suppose the assumptions of Theorem \ref{thm:linConvSGA-LSC} hold. In the special case where $\Lo=\mu_\Omega$, the function $h$ is quadratic, i.e., $h(x)=\tfrac{\Lo}{2}\|x\|^2 + \innprod{\nabla h(0)}{x} + h(0)$ by Corollary \ref{cor:LocalStrLipCont}. This yields that the optimal solution is given by $x^*= -\tfrac{\nabla h(0)}{\Lo}$. Consequently, this case was excluded from Theorem \ref{thm:linConvSGA-LSC}.\qed
\end{rem}

The following theorem investigates the complexity of the SGA under local Lipschitzness of $\nabla f$, ($\nu=1$), local strong convexity. For simplicity, let us define the following operator for $x,y>0$:
\begin{align*}
(x,y) \mapsto\MN(x, y):=\left\lceil\tfrac{\log(\varepsilon^{-1}) + \log\left(x\right)}{\log\big(y^{-1}\big)}+1\right\rceil.
\end{align*}

\begin{thm}[\textbf{Complexity analysis under local strong convexity}] 
\label{cor:ComAnaSGA-LSC}
    Let Assumption \ref{ass:LevelBound}, let $\nabla f$ be locally Lipschitz $(\nu=1)$, and let the function $f$ be locally strongly convex where $\mu_\Omega<\Lo$. If the sequence $\seq{x^k}$ is generated by Algorithm~\ref{alg:SGA} with dynamic or constant step-size $\alpha_k\in [\ov\alpha,\tfrac{1}{\Lo}]$ at which $0<\ov\alpha<\tfrac{1}{\Lo}$, for a given accuracy parameter $\varepsilon>0$, the following assertions hold:
    \begin{enumerate}
        \item\label{cor:ComAnaSGA-LSC-1}
            The number of iterations to guarantee $\|x^k-x^*\|\leq \varepsilon$, denoted by $\Nx(\varepsilon)$, satisfies $$\Nx(\varepsilon)\le \MN(\|x^0-x^*\|,\sqrt{1-q_0});$$
        \item\label{cor:ComAnaSGA-LSC-2}
            The number of iterations to guarantee $f(x^k)-f^*\leq \varepsilon$, denoted by $\Nf(\varepsilon)$, satisfies $$\Nf(\varepsilon)\le \MN(f(x^{0})-f^*,1-q_1);$$
        \item\label{cor:ComAnaSGA-LSC-3}
            The number of iterations to guarantee $\|\nabla f(x^k)\|\leq \varepsilon$, denoted by $\Nnf(\varepsilon)$, satisfies $$\Nnf(\varepsilon)\le \MN\Big(\sqrt{2\Lo^2\mu_\Omega^{-1}(f(x^{0})-f^*)},\sqrt{1-q_1}\Big);$$
    \end{enumerate}
    where $x^*$ is an optimal solution, $q_0:=\tfrac{2\ov{\alpha}\mu_\Omega \Lo}{\mu_\Omega+\Lo}$, and $q_1:=\tfrac{\mu_\Omega^2\ov{\alpha}^2}{4}$.
\end{thm}
\begin{proof}
    \ref{cor:ComAnaSGA-LSC-1} Setting $K_1:=\MN(\|x^0-x^*\|,\sqrt{1-q_0})$, it follows from Theorem \ref{thm:linConvSGA-LSC}~\ref{thm:linConvSGA-LSC-1} that
    \[
    \|x^{K_1}-x^*\|\le \big( 1-q_0\big)^{\tfrac{K_1}{2}} \|x^0-x^*\| \le \varepsilon.
    \]
    ensuring $\Nx(\varepsilon)\le K_1$.
    The proof of Assertions \ref{cor:ComAnaSGA-LSC-2} and \ref{cor:ComAnaSGA-LSC-3} is similar to one of Assertion \ref{cor:ComAnaSGA-LSC-1} applying Theorem \ref{thm:linConvSGA-LSC}.\qed
\end{proof}

\subsubsection{\textbf{Linear convergence under KL inequality}}

In this subsection we provide the linear convergence of the SAG under the assumption of the KL inequality. As a preliminary step, the following lemma shows that the KL exponent $\vartheta$ corresponding to the limiting point of the sequence $\seq{x_k}$, generated by Algorithm~\ref{alg:SGA}, satisfies the bound $\vartheta_\Omega\ge \tfrac{\nu}{1+\nu}$.

\begin{lem}\label{lem:Bound-KL}
    Let Assumption \ref{ass:LevelBound} hold and let the sequence $\seq{x^k}$ be generated by Algorithm~\ref{alg:SGA} with dynamic or constant step-size $\alpha_k \in [\ov\alpha , \tfrac{1}{\Lo}]$ where $0<\ov \alpha<\tfrac{1}{\Lo}$.
    If $f$ is satisfied the KL inequality, then $\vartheta_\Omega \ge \tfrac{\nu}{1+\nu}$.
\end{lem}
\begin{proof}
    By Theorems \ref{thm:SufDecSGA}~\ref{thm:SufDecSGA} and \ref{thm:GlobconvSGA}~\ref{thm:GlobconvSGA-3}, $\seq{x^k} \subseteq \Omega$ and $\displaystyle \lim_{k\to\infty} f(x^k) = f^*$. Hence,
    \begin{equation}\label{eq:KL-Local-Global}
        (f(x^k)-f^*)^{\vartheta_\Omega} \leq \rho_\Omega\|\nabla f(x^k)\|,\quad\quad \forall k\in \N_0,
    \end{equation}
    due to Theorem~\ref{thm:KL}~\ref{thm:KL-2}. Combining this inequality with inequality \eqref{eq:squarefk1fk-2} in Theorem \ref{thm:GlobconvSGA}~\ref{thm:GlobconvSGA-1}, yields that
    \begin{equation}\label{eq:Ineq-fun-SGA-KL}
        (f(x^k)-f^*)^{\vartheta_\Omega} \le \rho(\tfrac{1+\nu}{\Lo\ov\alpha^{1+\nu}})^{\tfrac{\nu}{1+\nu}} \big(f(x^k)-f(x^{k+1})\big)^{\tfrac{\nu}{1+\nu}} \le \rho(\tfrac{1+\nu}{\Lo\ov\alpha^{1+\nu}})^{\tfrac{\nu}{1+\nu}} \big(f(x^k)-f^*\big)^{\tfrac{\nu}{1+\nu}}, \quad\quad \forall k\in \N_0,
    \end{equation}
    leading to 
    \[
    (f(x^k)-f^*)^{\vartheta_\Omega-\tfrac{\nu}{1+\nu}} \le \rho(\tfrac{1+\nu}{\Lo\ov\alpha^{1+\nu}})^{\tfrac{\nu}{1+\nu}}, \quad\quad \forall k\in \N_0.
    \]
    This ensures $\vartheta_\Omega\ge \tfrac{\nu}{1+\nu}$ inasmuch as $f(x^k)-f^*\to 0$ as $k\to \infty$. This completes the proof.\qed
\end{proof}

According to Lemma \ref{lem:Bound-KL}, the condition $\vartheta_\Omega \ge \tfrac{\nu}{1+\nu}$ is necessary to ensure the possibility $f(x^k)-f^*\to 0$ as $k\to \infty$. The next theorem establishes the linear convergence of function value gap and norm of gradient under KL property.

\begin{thm}[\textbf{Linear convergence rate of function value and gradient under the KL inequality}]
\label{thm:linConv-Fun-SGA-KL}
    Let Assumption \ref{assumptionSG}  hold and 
    let the sequence $\seq{x^k}$ be generated by Algorithm~\ref{alg:SGA} with dynamic or constant step-size $\alpha_k \in [\ov\alpha , \tfrac{1}{\Lo}]$ where $0<\ov \alpha<\tfrac{1}{\Lo}$.
    If $f$ is satisfied the KL inequality, the following assertions hold:
    \begin{enumerate}
        \item\label{thm:linConv-Fun-SGA-KL-1} 
            If $\vartheta_\Omega= \tfrac{\nu}{1+\nu}$, the sequences $\seq{f(x^k)-f^*}$ and $\seq{\|\nabla f(x^k)\|}$ converges to $0$ Q-linearly and R-linearly, respectively, i.e.,
        \begin{equation}\label{eq:Lin-Fun-Inqu}
            f(x^{k+1})-f^* \leq \left(1-q_2\right) \big(f(x^{k})-f^*\big),\quad\quad\forall k\in \N_0,
        \end{equation}
         \begin{equation}\label{eq:Lin-Grad-Inqu}
            \|\nabla f(x^k)\| \leq \Big(\tfrac{f(x^{0})-f^*}{(1-\vartheta_\Omega)\Lo\ov\alpha^{\tfrac{1}{1-\vartheta_\Omega}}}\Big)^{\vartheta_\Omega} \left(1-q_2\right)^{k\vartheta_\Omega}, \quad\quad\forall k\in\N_0,
        \end{equation}
        where $q_2:=(1-\vartheta_\Omega)\Lo\ov\alpha^{\tfrac{1}{1-\vartheta_\Omega}}\rho^{-\tfrac{1}{\vartheta_\Omega}}$;
        \item\label{thm:linConv-Fun-SGA-KL-2} 
            If $\vartheta_\Omega >\tfrac{\nu}{1+\nu}$, there exists $\sigma>0$ such that 
        \begin{equation}\label{eq:SubLin-Fun-Inqu}
            0\leq f(x^k)-f^*\leq \sigma k^{-\tfrac{\nu}{\vartheta_\Omega (1+\nu)-\nu}},
        \end{equation}
        \begin{equation}\label{eq:SubLin-Grad-Inqu}
            \|\nabla f(x^k)\| \leq \big(\tfrac{(1+\nu)\sigma}{\Lo\ov\alpha^{1+\nu}}\big)^{\tfrac{\nu}{1+\nu}} k^{-\tfrac{\nu^2}{\big(\vartheta_\Omega (1+\nu)-\nu\big)(1+\nu)}},
        \end{equation}
        for $k\in\N_{0}$ sufficiently large.
    \end{enumerate}
\end{thm}
\begin{proof}
    Employing inequality \eqref{eq:Ineq-fun-SGA-KL}, we obtain
    \begin{align*}
         \big(f(x^k)-f^*\big)^{\tfrac{\vartheta_{\Omega}(1+\nu)}{\nu}} \le \tfrac{\rho^{\tfrac{1+\nu}{\nu}}(1+\nu)}{\Lo\ov\alpha^{1+\nu}}\big(f(x^k)-f(x^{k+1})\big) =\tfrac{\rho^{\tfrac{1+\nu}{\nu}}(1+\nu)}{\Lo\ov\alpha^{1+\nu}}\Big(\big(f(x^k)-f^*\big)- \big(f(x^{k+1})-f^*\big)\Big), \quad\quad \forall k\ge 0,
    \end{align*}
    which leads to a sequence of the form \eqref{eq:sk} with $s^k:=f(x^k)-f^*$, $\theta:=\tfrac{\vartheta_{\Omega}(1+\nu)}{\nu}$, and $\beta:= \tfrac{\rho^{\tfrac{1+\nu}{\nu}}(1+\nu)}{\Lo\ov\alpha^{1+\nu}}$. Applying Fact~\ref{fac:convRate1} to this setup yields the convergence rates \eqref{eq:Lin-Fun-Inqu} and \eqref{eq:SubLin-Fun-Inqu}. Moreover, according to Theorem \ref{thm:GlobconvSGA}~\ref{thm:GlobconvSGA-1} (inequality \eqref{eq:squarefk1fk-2}), it holds that
    \begin{equation*}\label{eq:Upper-Grad}
        \|\nabla f(x^k)\| \le \big(\tfrac{1+\nu}{\Lo\ov\alpha^{1+\nu}}\big)^{\tfrac{\nu}{1+\nu}} \big(f(x^k)-f(x^{k+1})\big)^{\tfrac{\nu}{1+\nu}}\le \big(\tfrac{1+\nu}{\Lo\ov\alpha^{1+\nu}}\big)^{\tfrac{\nu}{1+\nu}} \big(f(x^k)-f^*\big)^{\tfrac{\nu}{1+\nu}}, \quad\quad \forall k\in \N_0.
    \end{equation*}
    Thus, the convergence rates \eqref{eq:Lin-Grad-Inqu} and \eqref{eq:SubLin-Grad-Inqu} follow from \eqref{eq:Lin-Fun-Inqu} and \eqref{eq:SubLin-Fun-Inqu} together with the former inequality, adjusting the result.\qed
\end{proof}

Theorem \ref{thm:linConv-Itr-SGA-KL} presents the linear convergence of iterations under KL property.

\begin{thm}[\textbf{Linear convergence rate of iterations under the KL inequality}]
\label{thm:linConv-Itr-SGA-KL}
    Let Assumption \ref{assumptionSG} hold and 
    let the sequence $\seq{x^k}$ be generated by Algorithm~\ref{alg:SGA} with dynamic or constant step-size $\alpha_k \in [\ov\alpha , \tfrac{2\nu}{(1+\nu)\Lo^{\nicefrac{1}{\nu}}}]$ where $0<\ov \alpha<\tfrac{2\nu}{(1+\nu)\Lo^{\nicefrac{1}{\nu}}}$.
    If $f$ is satisfied the KL inequality, the following assertions hold:
    \begin{enumerate}
        \item\label{thm:linConv-Itr-SGA-KL-1}
            If $\vartheta_\Omega= \tfrac{\nu}{1+\nu}$, the sequence $\seq{\|x^k-x^*\|}$ converges R-linearly to $0$, i.e.,
        \begin{equation}\label{eq:Lin-Itr-Inqu}
            \|x^k-x^*\| \leq  \tfrac{2\vartheta_\Omega\rho_\Omega\big(f(x^0)-f^*\big)^{1-\vartheta_\Omega}}{(1-\vartheta_\Omega)^2\Lo^{\tfrac{1}{\vartheta_\Omega}}\ov\alpha^{\tfrac{1}{1-\vartheta_\Omega}}} (1-q_2)^{k(1-\vartheta_\Omega)},\quad\quad\forall k\in N_0,
        \end{equation}
        where $q_2:=(1-\vartheta_\Omega)\Lo\ov\alpha^{\tfrac{1}{1-\vartheta_\Omega}}\rho^{-\tfrac{1}{\vartheta_\Omega}}$ and $x^*\in \R^n$ is the limiting point of $\seq{x^k}$;
        \item\label{thm:linConv-Itr-SGA-KL-2}
        If $\vartheta_\Omega > \tfrac{\nu}{1+\nu}$, there exists $\sigma>0$ such that 
        \begin{equation}\label{eq:SubLin-Itr-Inqu}
            \|x^k-x^*\| \leq  \tfrac{2\nu\rho_\Omega\sigma^{1-\vartheta_\Omega}}{(1-\vartheta_\Omega)\Lo^{\tfrac{1+\nu}{\nu}}\ov\alpha^{1+\nu}} k^{-\tfrac{\nu(1-\vartheta_\Omega)}{\vartheta_\Omega (1+\nu)-\nu}},
        \end{equation}
        for $k\in\N_{0}$ sufficiently large.
    \end{enumerate}
\end{thm}
\begin{proof}
It follows from concavity of function $t \mapsto t^{1-\vartheta_\Omega}$ and inequality \eqref{eq:squarefk1fk-2} that
\begin{align*}
    \big(f(x^{k})-f^* \big)^{1-\vartheta_\Omega} - \big( f(x^{k+1})-f^*\big)^{1-\vartheta_\Omega} &\ge (1-\vartheta_\Omega) \big(f(x^{k})-f^* \big)^{-\vartheta_\Omega} \big(f(x^{k}) - f(x^{k+1}) \big)\\
    &\ge \tfrac{(1-\vartheta_\Omega)\Lo\ov\alpha^{1+\nu} }{1+\nu} \big(f(x^{k})-f^* \big)^{-\vartheta_\Omega} \|\nabla f(x^k)\|^{\tfrac{1+\nu}{\nu}}.
\end{align*}
Applying KL inequality the former inequality comes to
\[
\big(f(x^{k})-f^* \big)^{1-\vartheta_\Omega} - \big( f(x^{k+1})-f^*\big)^{1-\vartheta_\Omega} \ge \tfrac{(1-\vartheta_\Omega)\Lo\ov\alpha^{1+\nu} }{(1+\nu)\rho_\Omega} \|\nabla f(x^k)\|^{\tfrac{1}{\nu}}, \quad\quad\forall k\in\N_0. 
\]
Using this inequality, we obtain
\begin{align*}\label{eq:Upper-Itr}
    \|x^k-x^*\| &\le \sum_{i\ge k} \|x^{i}-x^{i+1}\| = \sum_{i\ge k} \alpha_i\|\nabla f(x^i)\|^{\tfrac{1}{\nu}} \le \tfrac{2\nu}{(1+\nu)\Lo^{\nicefrac{1}{\nu}}}\sum_{i\ge k} \|\nabla f(x^i)\|^{\tfrac{1}{\nu}}\\
    &\le \tfrac{2\nu}{(1+\nu)\Lo^{\nicefrac{1}{\nu}}} \tfrac{(1+\nu)\rho_\Omega}{(1-\vartheta_\Omega)\Lo\ov\alpha^{1+\nu}} \sum_{i\ge k}\Big(\big(f(x^i)-f^*\big)^{1-\vartheta_\Omega}-\big(f(x^{i+1})-f^*\big)^{1-\vartheta_\Omega}\Big)\\
    &= \tfrac{2\nu\rho_\Omega}{(1-\vartheta_\Omega)\Lo^{\tfrac{1+\nu}{\nu}}\ov\alpha^{1+\nu}} \big(f(x^k)-f^{*}\big)^{1-\vartheta_\Omega}, \quad\quad \forall k\in N_0,
\end{align*}
ensuring \eqref{eq:Lin-Itr-Inqu} and \eqref{eq:SubLin-Itr-Inqu} employing \eqref{eq:Lin-Fun-Inqu} and \eqref{eq:SubLin-Fun-Inqu}. This completes the proof.\qed
\end{proof}

The SGA exhibits notable flexibility in accommodating various smoothness settings. When the gradient $\nabla f$ is locally Lipschitz, it is also locally $\nu$-H\"{o}lder with any $\nu\in (0,1)$. Consequently, if $f$ satisfies KL inequality, then by setting $\nu:=\tfrac{\vartheta_\Omega}{1-\vartheta_\Omega}$, the method attains a linear rate of convergence. Therefore, for functions satisfying the KL inequality with locally Lipschitz continuous gradients, the SGA achieves linear convergence for any KL exponent $\vartheta \in (0,1)$. 

Theorem \ref{thm:linConv-Com-SGA-KL} analyzes the complexity of the SGA under local $\nu$-H\"{o}lder continuity of $\nabla f$ and KL inequality.

\begin{thm}[\textbf{Complexity analysis under the KL inequality}]
\label{thm:linConv-Com-SGA-KL}
    Let Assumption \ref{assumptionSG} hold and 
    let the sequence $\seq{x^k}$ be generated by Algorithm~\ref{alg:SGA} with dynamic or constant step-size $\alpha_k \in [\ov\alpha , \tfrac{1}{\Lo}]$ where $0<\ov \alpha<\tfrac{1}{\Lo}$.
    If $f$ is satisfied the KL inequality with exponent $\vartheta_\Omega= \tfrac{\nu}{1+\nu}$, the following assertions hold:
    \begin{enumerate}
        \item\label{thm:linConv-Com-SGA-KL-1}
            The number of iterations to guarantee $f(x^k)-f^*\leq \varepsilon$, $\Nf(\varepsilon)$, satisfies $$\Nf(\varepsilon)\le \MN\big(f(x^{0})-f^*,1-q_2\big);$$
        \item\label{thm:linConv-Com-SGA-KL-2}
            The number of iterations to guarantee $\|\nabla f(x^k)\|\leq \varepsilon$, $\Nnf(\varepsilon)$, satisfies $$\Nnf(\varepsilon)\le \MN\bigg(\Big(\tfrac{f(x^{0})-f^*}{(1-\vartheta_\Omega)\Lo\ov\alpha^{\tfrac{1}{1-\vartheta_\Omega}}}\Big)^{\vartheta_\Omega},(1-q_2)^{\vartheta_\Omega}\bigg);$$
        \item\label{thm:linConv-Com-SGA-KL-3}
            If $\alpha_k \in \big[\ov\alpha , 2\vartheta_\Omega \Lo^{-\tfrac{1-\vartheta_\Omega}{\vartheta_\Omega}}\big]$ where $0<\ov \alpha<2\vartheta_\Omega \Lo^{-\tfrac{1-\vartheta_\Omega}{\vartheta_\Omega}}$, the number of iterations to guarantee $\|x^k-x^*\|\leq \varepsilon$, $\Nx(\varepsilon)$, satisfies $$\Nx(\varepsilon)\le \MN\bigg(\tfrac{2\vartheta_\Omega\rho_\Omega\big(f(x^0)-f^*\big)^{1-\vartheta_\Omega}}{(1-\vartheta_\Omega)^2\Lo^{\tfrac{1}{\vartheta_\Omega}}\ov\alpha^{\tfrac{1}{1-\vartheta_\Omega}}},(1-q_2)^{1-\vartheta_\Omega}\bigg);$$
    \end{enumerate}
    where $x^*$ is an optimal solution, $q_2:=(1-\vartheta_\Omega)\Lo\ov\alpha^{\tfrac{1}{1-\vartheta_\Omega}}\rho^{-\tfrac{1}{\vartheta_\Omega}}$.
\end{thm}
\begin{proof}
    The complexity bounds are established analogously to that in Theorem~\ref{cor:ComAnaSGA-LSC}.\qed
\end{proof}

\section{AdaSGA (adaptive scaled gradient algorithm)} \label{sec:AdaSGA}
Let $\func{f}{\R^n}{\R}$ be a smooth convex function with a locally Lipschitz gradient (Assumption \ref{assumptionSG}). Considering a special class of these functions satisfying level set boundedness, we have already studied the scaled gradient method (SGA), which may lead to a slow convergence rate if the Lipschitz modulus $\Lo>0$ is big. Moreover, it is not easy to compute this Lipschitz modulus in practice.

\begin{rem}\label{rem:adaSG}
    Given an initial point $x^0\in\R^n$ and parameters $\gamma_0, \alpha_0, \theta_0 >0$, let us consider the compact convex set $\W:=\overline{B}(0;R)$ where $R>3\sqrt{\eta}+\dist(x^0;\X^*)+\|x^0\|$ and
\begin{equation*}\label{eq:eta}
    \eta:= \dist^2(x^0;\X^*)+2\alpha_0^2 \gamma_0^2\|\nabla f(x^0)\|^2+2\alpha_0\gamma_0\theta_0\big(f(x^0)-f^*\big).
\end{equation*}
Inasmuch as $\nabla f$ is locally Lipschitz, the function $f$ is $L_\W$-smooth on $\W=\overline{B}(0;R)$.\qed
\end{rem}

In this section, we investigate an adaptive variant of SGA (called AdaSGA; see Algorithm~\ref{alg:AdaSGDA}) that does not need to know the modulus $L_\W$ for the implementation and may attain a much larger step-size, which consequently leads to a faster convergence to a solution of \eqref{eq:p}. Next, for a generated sequence $\seq{x^k}$ by AdaSGA, we analyze the convergence of the sequences the iterates distance $\seq{\|x^k - x^*\|}$, the function value gap $\seq{f(x^k)-f^*}$, and the gradient norm $\seq{\|\nabla f(x^k)\|}$.
To do so, we note that
\begin{align}\label{eq:Lipk}
    L_k:=\tfrac{\|\nabla f(x^k)-\nabla f(x^{k-1})\|}{\|x^k-x^{k-1}\|},
\end{align}
is a local approximation of the Lipschitz modulus $L_\W$ that satisfies $L_k\leq L_\W$. We use this $L_k$ to approximate $L_\W$ that will lead to possible larger step-sizes given by
\begin{equation}\label{eq:alphak}
    \alpha_k:=\min\set{\tfrac{\alpha_{k-1} \gamma_{k-1}}{\gamma_k}\sqrt{2(1-\omega^2)+\tfrac{\theta_{k-1}}{\tau}}, \tfrac{\omega}{\gamma_kL_k}},\quad\quad\forall k\in\N,
\end{equation}
where
\begin{equation}\label{eq:thetak}
    \theta_{k}:=\tfrac{\alpha_{k}\gamma_{k}}{\alpha_{k-1}\gamma_{k-1}},\quad\quad\forall k\in\N,
\end{equation}
$\alpha_0>0$, $\theta_0 >0$, $0<\gamma_{\min}\leq\gamma_k\leq \gamma_{\max}$, $\tau\ge 1$, and $0<\omega\le\tfrac{1}{\sqrt{2}}$.

\vspace{4mm}
\RestyleAlgo{boxruled}
\begin{algorithm}[H]
\DontPrintSemicolon
\KwIn{$x^0\in\R^ n$,~ $0<\gamma_{\min}\leq\gamma_k\leq \gamma_{\max}$,~ $\alpha_0>0$,~$\theta_0 >0$, $\tau\ge 1$, $0<\omega\le\tfrac{1}{\sqrt{2}}$;}
\Begin{
    $k:=0$;
    
    \While{
        the stopping criteria does not hold
    }{
    
    Set $x^{k+1}=x^k-\alpha_k \gamma_k \nabla f(x^k)$;\\[1mm]
    Compute $\alpha_{k+1}$ and $\theta_{k+1}$ by \eqref{eq:alphak} and \eqref{eq:thetak}, respectively, and set $k=k+1$;
    }
    $x^{\mathrm{best}}=x^k$;\;
}
\KwOut{$x^{\mathrm{best}}$} 
\caption{AdaSGA: Adaptive Scaled Gradient Algorithm \label{alg:AdaSGDA}}
\end{algorithm}

\vspace{4mm}
Let us begin with the subsequent result establishing the Lyapunov decrease condition.

\begin{thm}[\textbf{Lyapunov decrease condition}]
\label{thm:Suf-Dec-Lya}
    Let $x\in \R^n$. If the sequence $\seq{x^k}$ is generated by Algorithm~\ref{alg:AdaSGDA}, one has
    \begin{equation}\label{eq:normXk1XsUpBoundomega}
    \begin{array}{ll}
         \|x^{k+1}-x\|^2&+\tfrac{\omega^2}{1-\omega^2} \|x^{k+1}-x^k\|^2+2\alpha_k \gamma_k \left(1+\tfrac{\theta_{k}}{2(1-\omega^2)}\right)\big(f(x^k)-f(x)\big)\\[2mm]
        &\leq \|x^k-x\|^2+\tfrac{\omega^2}{1-\omega^2} \|x^k-x^{k-1}\|^2+\tfrac{\alpha_k\gamma_k\theta_{k}}{1-\omega^2} \left(f(x^{k-1})-f(x)\right), \quad \quad \forall k\in\N.
    \end{array}
    \end{equation}
    In addition, the Lyapunov function $\func{\eL}{\R^n\times\R^n}{\R}$ given by
    \begin{align*}
        \eL(x^{k},x^{k-1}):=\|x^{k}-x\|^2+ \tfrac{\omega^2}{1-\omega^2}\|x^{k}-x^{k-1}\|^2+\tfrac{\alpha_k\gamma_k\theta_{k}}{1-\omega^2} \big(f(x^{k-1})-f(x)\big),
    \end{align*}
    is nonincreasing, i.e., $\eL(x^{k+1},x^k) \leq \eL(x^k,x^{k-1})$.
\end{thm}
\begin{proof}
    It follows from the convexity of $f$ that
    \begin{align}
        \|x^{k+1}-x\|^2
            &= \|x^k-x\|^2+2\alpha_k\gamma_k\innprod{\nabla f(x^k)}{x-x^k}+\|x^{k+1}-x^k\|^2\label{eq:normXk1XsUpBound11100}\\
            &\leq \|x^k-x\|^2+2\alpha_k \gamma_k \big(f(x)-f(x^k)\big)+\|x^{k+1}-x^k\|^2, \quad\quad\forall k\in \N_0.\label{eq:normXk1XsUpBound111}
    \end{align}
    Let us consider $k\in \N$. In each case, we find an upper bound for the term $\|x^{k+1}-x^k\|^2$ to reach out the desired inequality.
    It is clear that
    \[
        \|\nabla f(x^{k-1})\|^2 -\innprod{\nabla f(x^k)}{\nabla f(x^{k-1})}=\tfrac{1}{\alpha_{k-1}\gamma_{k-1}} \innprod{\nabla f(x^{k-1})-\nabla f(x^k)}{x^{k-1}-x^k}\geq 0.
    \]
    Using this inequality, the step-size condition $\alpha_k\gamma_kL_k \leq \omega<1$, and the convexity of $f$, we deduce
    \begin{equation*}
        \begin{split}
            \|x^{k+1}-x^k\|^2 &= \alpha_k^2\gamma_k^2 \|\nabla f(x^k)-\nabla f(x^{k-1})\|^2-\alpha_k^2\gamma_k^2 \|\nabla f(x^{k-1})\|^2+2\alpha_k^2\gamma_k^2\innprod{\nabla f(x^k)}{\nabla f(x^{k-1})}\\
            &\leq \alpha_k^2\gamma_k^2 L_k^2 \|x^k-x^{k-1}\|^2+\alpha_k\gamma_k\theta_{k} \innprod{\nabla f(x^k)}{x^{k-1}-x^k}\\
            &\leq \omega^2 \|x^k-x^{k-1}\|^2+\alpha_k\gamma_k \theta_{k} \left(f(x^{k-1})-f(x^k)\right).
        \end{split}
    \end{equation*}
    which implies
    \begin{equation}\label{eq:normXk1XkUpBound111}
        \begin{split}
            \|x^{k+1}-x^k\|^2 &\leq \tfrac{\omega^2}{1-\omega^2} \|x^k-x^{k-1}\|^2-\tfrac{\omega^2}{1-\omega^2} \|x^{k+1}-x^k\|^2+\tfrac{\alpha_k\gamma_k\theta_{k}}{1-\omega^2} \left(f(x^{k-1})-f(x^k)\right).
        \end{split}
    \end{equation}
    Substituting \eqref{eq:normXk1XkUpBound111} into \eqref{eq:normXk1XsUpBound111}, we come to
    \begin{equation*}
        \begin{split}
            \|x^{k+1}-x\|^2 &\leq \|x^k-x\|^2+2\alpha_k \gamma_k (f(x)-f(x^k))+\tfrac{\omega^2}{1-\omega^2} \|x^k-x^{k-1}\|^2\\
            &~~~-\tfrac{\omega^2}{1-\omega^2} \|x^{k+1}-x^k\|^2+\tfrac{\alpha_k\gamma_k\theta_{k}}{1-\omega^2} \left(f(x^{k-1})-f(x^k)\right),
        \end{split}
    \end{equation*}
    leading to
    \begin{equation*}
        \begin{split}
            \|x^{k+1}-x\|^2&+\tfrac{\omega^2}{1-\omega^2} \|x^{k+1}-x^k\|^2+2\alpha_k \gamma_k \left(1+\tfrac{\theta_{k}}{2(1-\omega^2)}\right)(f(x^k)-f(x))\\
            &\leq \|x^k-x\|^2+\tfrac{\omega^2}{1-\omega^2} \|x^k-x^{k-1}\|^2+\tfrac{\alpha_k\gamma_k\theta_{k}}{1-\omega^2} \left(f(x^{k-1})-f(x)\right).
        \end{split}
    \end{equation*}
    verifying \eqref{eq:normXk1XsUpBoundomega}. Moreover, by combining
    \begin{align*}
        \tfrac{\alpha_{k+1}\gamma_{k+1}\theta_{k+1}}{1-\omega^2} \leq 2\alpha_{k} \gamma_{k} \left(1+\tfrac{\theta_{k}}{2\tau(1-\omega^2)}\right)\leq 2\alpha_{k} \gamma_{k} \left(1+\tfrac{\theta_{k}}{2(1-\omega^2)}\right),
    \end{align*}
    i.e., $\alpha_{k+1}\leq \tfrac{\alpha_{k} \gamma_{k}}{\gamma_{k+1}}\sqrt{2(1-\omega^2)+\frac{\theta_{k}}{\tau}}$ and $\tau\ge 1$, with \eqref{eq:normXk1XsUpBoundomega} ensures the nonincreasing behavior of Lyapunov function and completes the proof.\qed
\end{proof}

By applying the nonincreasing property of the sequence $\seqone{\eL(x^{k},x^{k-1})}$, we proceed to analyze the global convergence.

\begin{thm}[\textbf{Global convergence of AdaSGA}]
\label{thm:Glo-Conv-AdaSGA}
    If the sequence $\seq{x^k}$ is generated by Algorithm~\ref{alg:AdaSGDA}, the following statements hold:
    \begin{enumerate}
        \item\label{thm:Glo-Conv-AdaSGA-1}
            The sequence $\seq{x^k}$ is bounded, i.e., $\seq{x^k}\subseteq \interior\W=B(0;R)$ where $R>\sqrt{\eta}+\dist(x^0;\X^*)+\|x^0\|$ and the constant $\eta$ is defined in Remark~\ref{rem:adaSG};
        \item\label{thm:Glo-Conv-AdaSGA-2}
            We have $\sum_{k\ge 0} \|\nabla f(x^k)\|^2 <\infty$ and $\displaystyle\lim_{k\to\infty} \|\nabla f(x^k)\|=0$;
        \item\label{thm:Glo-Conv-AdaSGA-3}
            The sequence $\seq{x^k}$ converges to an optimal solution $x^*\in \interior\W=B(0;R)$.
    \end{enumerate}
\end{thm}

\begin{proof}
    Let us assume that $x^0\not\in \X^*$ and consider $x_*^0 := \proj_{\X^*}(x^0)$.\\
    \ref{thm:Glo-Conv-AdaSGA-1}
    Based on Theorem \ref{thm:Suf-Dec-Lya}, the sequence $\seqone{\eL(x^k,x^{k-1})}$ is nonincreasing, i.e.,
    \[\eL(x^{k},x^{k-1}) \leq \eL(x^{k-1},x^{k-2})\leq \ldots \leq \eL(x^1,x^0),\]
    and as a consequence,
    \begin{align*}
        \|x^{k}-x_*^0\|^2\le \|x^1-x_*^0\|^2+\tfrac{\omega^2}{1-\omega^2}\|x^1-x^0\|^2+\tfrac{\alpha_1 \gamma_1\theta_1}{1-\omega^2}\big(f(x^0)-f(x_*^0)\big).
    \end{align*}
    Applying \eqref{eq:normXk1XsUpBound111} with $k=0$ and $x=x_*^0$ yields
    \begin{align}\label{eq:Uperx1x*}
        \|x^1-x_*^0\|^2 \leq \|x^0-x_*^0\|^2-2\alpha_0 \gamma_0 (f(x^0)-f^*)+\|x^1-x^0\|^2.
    \end{align}
    Combining the last two inequalities ensures
    \begin{align}\label{eq:Uperxkx*0}
    \|x^k-x_*^0\|^2\leq \|x^0-x_*^0\|^2+(1+\tfrac{\omega^2}{1-\omega^2})\|x^1-x^0\|^2+(\tfrac{\alpha_1 \gamma_1\theta_1}{1-\omega^2}-2\alpha_0 \gamma_0) (f(x^0)-f^*)\le\eta,
    \end{align}
    for the constant $\eta$ defined in Remark~\ref{rem:adaSG}.
    Consequently,
    \begin{equation}\label{eq:Upp-Bou-xk}
        \|x^k\| \leq \|x^k-x_*^0\| + \|x_*^0-x^0\| + \|x^0\| \le \sqrt{\eta} + \dist(x^0;\X^*) + \|x^0\| < R,
    \end{equation}
    which confirms that $\seq{x^k}\subseteq B(0;R)$.\\
    \ref{thm:Glo-Conv-AdaSGA-2}
    Due to the Assertion $(a)$, $x^k\in \interior \W=B(0;R)$. Furthermore,
    $$\|x_*^0\|\le \|x_*^0-x^0\|+\|x^0\| = \dist(x^0;\X^*) + \|x^0\| < R,$$
    ensuring $x_*^0\in \interior\W$. From $L_\W$-smoothness of function $f$ on $\W$, one has
    $$\left\|x_*^0-\tfrac{1}{L_\W}\big(\nabla f(x_*^0)-\nabla f(x^k)\big)\right\|\le \|x_*^0\| + \tfrac{1}{L_\W}\|\nabla f(x^k)-\nabla f(x_*^0)\| \leq \dist(x^0;\X^*) + \|x^0\| + \|x^k-x_*^0\| < R,$$
    i.e., $x_*^0-\tfrac{1}{L_\W}\big(\nabla f(x_*^0)-\nabla f(x^k)\big) \in \interior\W$. Thus, by virtue of Proposition \ref{pro:HSmoothChar}~\ref{pro:HSmoothChar-3}, it holds that
    \begin{equation*}\label{eq:Ext-Conv-Ineq}
        f(x_*^0)\geq f(x^k)+\innprod{\nabla f(x^k)}{x_*^0-x^k}+\tfrac{1}{2L_\W}\|\nabla f(x_*^0)-\nabla f(x^k)\|^2.
    \end{equation*}
    Applying the above inequality in \eqref{eq:normXk1XsUpBound11100} with $x=x_*^0$ implies that
    \begin{equation}\label{eq:normXk1XsUpBound3}
    \|x^{k+1}-x_*^0\|^2 \leq \|x^k-x_*^0\|^2+2\alpha_k \gamma_k \big(f^*-f(x^k)\big)- \tfrac{\alpha_k \gamma_k}{L_\W}\|\nabla f(x^k)\|^2 + \|x^{k+1}-x^k\|^2.
    \end{equation}
    Substituting \eqref{eq:normXk1XkUpBound111} into \eqref{eq:normXk1XsUpBound3} and utilizing the inequality $\alpha_{k+1}\gamma_{k+1}\leq \alpha_{k}\gamma_{k}\sqrt{2(1-\omega^2)+\tfrac{\theta_k}{\tau}}$ lead to
    \begin{equation*}
        \begin{split}
            \tfrac{\alpha_k \gamma_k}{L_\W}\|\nabla f(x^k)\|^2 \le \|x^k-x_*^0\|^2 &- \|x^{k+1}-x_*^0\|^2 + \tfrac{\omega^2}{1-\omega^2}\big(\|x^k-x^{k-1}\|^2-\|x^{k+1}-x^k\|^2\big)\\
            &+\tfrac{\alpha_k\gamma_k\theta_{k}}{1-\omega^2} \left(f(x^{k-1})-f^*\right)-\tfrac{\alpha_{k+1}\gamma_{k+1}\theta_{k+1}}{1-\omega^2} \left(f(x^{k})-f^*\right).
        \end{split}
    \end{equation*}
    Summing both sides of this inequality from $k=1$ to $k=N$ results in
    \begin{equation}\label{eq:upergrad}
        \begin{array}{ll}
             \displaystyle\sum_{k=1}^N \tfrac{\alpha_k\gamma_k}{L_\W}\|\nabla f(x^k)\|^2 &\le  \|x^1-x_*^0\|^2-\|x^{N+1}-x_*^0\|^2 +\tfrac{\omega^2}{1-\omega^2}\big(\|x^{1}-x^0\|^2-\|x^{N+1}-x^N\|^2\big)\\
             &\hspace{2.1cm}+ \frac{\alpha_{1} \gamma_{1} \theta_{1}}{1-\omega^2}\big(f(x^{0})-f^*\big)-\frac{\alpha_{N+1} \gamma_{N+1} \theta_{N+1}}{1-\omega^2}\big(f(x^N)-f^*\big)\\[2mm]
            &\le \|x^1-x_*^0\|^2+ \tfrac{\omega^2}{1-\omega^2}\|x^{1}-x^0\|^2+\frac{\alpha_{1} \gamma_{1} \theta_{1}}{1-\omega^2}\big(f(x^{0})-f^*\big) \\[2mm]
            &\le \|x^0-x_*^0\|^2 + (1+\tfrac{\omega^2}{1-\omega^2}) \|x^1-x^0\|^2 +(\frac{\alpha_{1} \gamma_{1} \theta_{1}}{1-\omega^2}-2\alpha_0 \gamma_0) (f(x^0)-f(x^*))\le\eta,
        \end{array}
    \end{equation}
    where the third inequality comes from \eqref{eq:Uperx1x*}.
    Now, it is enough to verify that the sequence $\seqone{\alpha_k\gamma_k}$ is bounded from below, i.e.,
    \begin{equation}\label{eq:LowerAlphaGamma}
        \min\left\{\alpha_0\gamma_0, \tfrac{\omega}{L_\W}\right\}\le \alpha_k\gamma_k,\quad\quad k\in\N_0.
    \end{equation}
    Then, \eqref{eq:upergrad} comes to
    \begin{equation}\label{eq:UpperNabla}
        \sum_{k=1}^N \|\nabla f(x^k)\|^2 \le \tfrac{L_\W^2 \eta}{\min\left\{L_\W\alpha_0\gamma_0, \omega\right\}},
    \end{equation}
    justifying $\sum_{k\ge 0} \|\nabla f(x^k)\|^2 <\infty$ and $\displaystyle\lim_{k\to\infty} \|\nabla f(x^k)\|=0$.
    We show inequality \eqref{eq:LowerAlphaGamma} by induction. For the case $k=0$, it is obvious. Assume that \eqref{eq:LowerAlphaGamma} holds for some $k\in \N_0$. We obtain
    $$\alpha_{k+1}\gamma_{k+1} = \min\left\{\alpha_k\gamma_k \sqrt{2(1-\omega^2)+\tfrac{\theta_k}{\tau}}, \tfrac{\omega}{L_{k+1}}\right\}\ge \min\left\{\alpha_k\gamma_k , \tfrac{\omega}{L_{k+1}}\right\} \ge \min\left\{\alpha_0\gamma_0, \tfrac{\omega}{L_\W}\right\}.$$
    \ref{thm:Glo-Conv-AdaSGA-3} Inasmuch as the sequence $\seq{x^k}$ is bounded and $\lim_{k\to\infty}\|\nabla f(x^k)\|= 0$, all cluster points of $\seq{x^k}$ are optimal solutions. The monotonicity of Lyapunov function implies that for any $x\in \X^*$,
    \[
    \|x^{k+1}- x\|^2 + a_{k+1} \leq \|x^{k}- x\|^2 + a_{k},
    \]
    where $a_k:=\tfrac{\omega^2}{1-\omega^2}\|x^k-x^{k-1}\|^2+ \frac{\alpha_k\gamma_k\theta_{k}}{1-\omega^2} \left(f(x^{k-1})-f^*\right)$. This inequality ensures that $x^k\to x^*$ for some $x^*\in \X^*$, due to \cite[Lemma 2]{malitsky2019adaptive}.
    Furthermore, it follows from \eqref{eq:Upp-Bou-xk} that $\|x^*\|\le \sqrt{\eta} + \dist(x^0;\X^*) + \|x^0\| < R,$ completing the proof.\qed
\end{proof}

The next theorem investigates the convergence rate and complexity analysis of AdaSGA to guarantee $\|\nabla f(x^k)\|\leq \varepsilon$ and $f(x^k)-f^*\leq \varepsilon$, for a specified accuracy
parameter $\varepsilon>0$.

\begin{thm}[\textbf{Convergence rate and complexity analysis of AdaSGA}] \label{thm:ConRat-Com-AdaSGA}
    Let $\W=\ov B(0;R)$ where $R>\sqrt{\eta}+\dist(x^0;\X^*)+\|x^0\|$ and the constant $\eta$ is defined in Remark~\ref{rem:adaSG}.
    If the sequence $\seq{x^k}$ is generated by Algorithm~\ref{alg:AdaSGDA}, the following statements hold:
    \begin{enumerate}
        \item\label{thm:ConRat-Com-AdaSGA-1}
            one has
            \[
            \min_{1\le k\le N} \|\nabla f(x^k)\| \le \sqrt{\tfrac{L_\W^2 \eta}{N\min\left\{L_\W\alpha_0\gamma_0, \omega\right\}}},\quad\quad \forall N\in \N.
            \]
            Moreover, for a give accuracy parameter $\varepsilon>0$, the number of iterations to guarantee $\|\nabla f(x_{k})\| \leq \varepsilon$, $\Nnf(\varepsilon)$, satisfies
            \begin{equation}\label{eq:compBoundAdaSGA2}
                \Nnf(\varepsilon)\le \K_0:=\left\lceil 1+\tfrac{L_\W^2 \eta}{\varepsilon^2\min\left\{L_\W\alpha_0\gamma_0, \omega\right\}}\right\rceil;
            \end{equation}
        \item\label{thm:ConRat-Com-AdaSGA-2}
            One has
            \[
            f^*_N-f^* \leq  \tfrac{L_\W\big((R+\|x^0\|)^2+\eta\big)}{2N\min\left\{L_\W\alpha_0\gamma_0, \omega\right\}},\quad\quad \forall N\in \N,
            \]
            where $f^*_N:=\min \{f(x^k): 0\le k\le N\}$.
            Moreover, for a give accuracy parameter $\varepsilon>0$, the number of iterations to guarantee $f(x^k)-f^*\leq \varepsilon$, $\Nf(\varepsilon)$, satisfies
            \begin{equation}\label{eq:compBoundAdaSGA1}
                \Nf(\varepsilon)\le \ov\K_0:=\left\lceil 1+\tfrac{L_\W\big((R+\|x^0\|)^2+\eta\big)}{2\varepsilon\min\left\{L_\W\alpha_0\gamma_0, \omega\right\}}\right\rceil.
            \end{equation}
    \end{enumerate}
\end{thm}
\begin{proof}
    \ref{thm:ConRat-Com-AdaSGA-1} The bound for norm of gradient is concluded from inequality \eqref{eq:UpperNabla}. Additionally, the proof of $\Nnf(\varepsilon)\le\K_0$ is similar to one in Theorem \ref{thm:ConvRat-Com-SGA}.\\
    \ref{thm:ConRat-Com-AdaSGA-2} By Theorem \ref{thm:Glo-Conv-AdaSGA}~\ref{thm:Glo-Conv-AdaSGA-3}, $x^k\to x^*$ for some $x^*\in \X^*$. Now, taking the sum from both sides of inequality \eqref{eq:normXk1XsUpBoundomega} with $x=x^*$ from $k=1$ to $k=N$ leads to
    \begin{equation}\label{eq:k1-to-kN-dif-Itr}
    \begin{split}
        &2\sum_{k=1}^{N} \left(\alpha_k\gamma_k\big(1+\tfrac{\theta_k}{2(1-\omega^2)}\big)\big(f(x^k)-f^*\big)-\tfrac{\alpha_{k}\gamma_{k}\theta_{k}}{2(1-\omega^2)}\big(f(x^{k-1})-f^*\big)\right)\\
        \le &\|x^1-x^*\|^2 - \|x^{N+1}-x^*\|^2 +\tfrac{\omega^2}{1-\omega^2}\big(\|x^1-x^0\|^2-\|x^{N+1}-x^N\|^2\big)\\
        \le & \|x^1-x^*\|^2 + \tfrac{\omega^2}{1-\omega^2}\|x^1-x^0\|^2.
    \end{split}
    \end{equation}
    On the other hand, the left side of inequality \eqref{eq:k1-to-kN-dif-Itr} is lower bounded as
    \begin{equation}\label{eq:k1-to-kN-dif-fun}
    \begin{split}
        &2\sum_{k=1}^{N} \left(\alpha_k\gamma_k\big(1+\tfrac{\theta_k}{2(1-\omega^2)}\big)\big(f(x^k)-f^*\big)-\tfrac{\alpha_{k}\gamma_{k}\theta_{k}}{2(1-\omega^2)}\big(f(x^{k-1})-f^*\big)\right)\\
        = &2\sum_{k=1}^{N-1} \bigg(\Big(\alpha_k\gamma_k\big(1+\tfrac{\theta_k}{2(1-\omega^2)}\big)-\tfrac{\alpha_{k+1}\gamma_{k+1}\theta_{k+1}}{2(1-\omega^2)}\Big)\big(f(x^{k})-f^*\big)\bigg)\\
        &\hspace{3cm}+ 2\alpha_{N}\gamma_{N}\big(1+\tfrac{\theta_{N}}{2(1-\omega^2)}\big)\big(f(x^{N})-f^*\big) - \tfrac{\alpha_{1}\gamma_{1}\theta_{1}}{1-\omega^2}\big(f(x^{0})-f^*\big)\\
        \ge &2\min_{1\le k\le N}\left\{f(x^k)-f^*\right\}\sum_{k=1}^{N-1} \Big(\alpha_k\gamma_k\big(1+\tfrac{\theta_k}{2(1-\omega^2)}\big)-\tfrac{\alpha_{k+1}\gamma_{k+1}\theta_{k+1}}{2(1-\omega^2)}\Big)\\
        &\hspace{3cm}+ 2\alpha_{N}\gamma_{N}\big(1+\tfrac{\theta_{N}}{2(1-\omega^2)}\big)\min_{1\le k\le N}\left\{f(x^k)-f^*\right\} - \tfrac{\alpha_{1}\gamma_{1}\theta_{1}}{1-\omega^2}\big(f(x^{0})-f^*\big)\\
        = &2 \big(f^*_N -f^*\big) \sum_{k=1}^{N} (\alpha_k\gamma_k) - \tfrac{\alpha_{1}\gamma_{1}\theta_{1}}{1-\omega^2}\big(f(x^{0})-f^*_N\big).
    \end{split}
    \end{equation}
    Combining \eqref{eq:k1-to-kN-dif-Itr} and \eqref{eq:k1-to-kN-dif-fun} and applying \eqref{eq:normXk1XsUpBound111} with $k=0$ and $x=x^*$, we obtain
    \begin{align*}
        2 \big(f^*_N -f^*\big) \sum_{k=1}^{N} \alpha_k\gamma_k
        &\leq \|x^1-x^*\|^2+\tfrac{\omega^2}{1-\omega^2}\|x^1-x^0\|^2+\tfrac{\alpha_1\gamma_1\theta_1}{1-\omega^2}\big(f(x^0)-f^*_N\big)\\
        &\leq \|x^0-x^*\|^2+ (1+\tfrac{\omega^2}{1-\omega^2})\|x^1-x^0\|^2+(\tfrac{\alpha_1\gamma_1\theta_1}{1-\omega^2} -2\alpha_0 \gamma_0) \big(f(x^0)-f^*\big)\le (R+\|x^0\|)^2+\eta,
    \end{align*}
    leading to
    \[
    f^*_N-f^* \leq \tfrac{\eta}{2\sum_{k=1}^{N} \alpha_k\gamma_k} \le \tfrac{L_\W\big((R+\|x^0\|)^2+\eta\big)}{2N\min\left\{L_\W\alpha_0\gamma_0, \omega\right\}},
    \]
    using \eqref{eq:LowerAlphaGamma}.
    Furthermore, $\Nf(\varepsilon)\le\ov\K_0$ is concluded from the former inequality together with \eqref{eq:compBoundAdaSGA1}. This completes the proof.
    \qed
\end{proof}

Theorem~\ref{thm:ConRat-Com-AdaSGA} establishes the sublinear convergence rate $\mathcal{O}(k^{-1})$ and the corresponding iteration complexity $\mathcal{O}(\varepsilon^{-1})$ for the sequence of function value gap $\seq{f(x^k)-f^*}$. In what follows, we strengthen the assumptions by imposing either local strong convexity or KL property, under which we demonstrate the linear convergence of the method. We first examine the case where the objective function $f$ is locally strongly convex and establish the linear convergence of the AdaSGA under this setting.

\begin{thm}[\textbf{Linear convergence rate of Lyapunov function under local strong convexity}]
\label{thm:linConv-Lya-STR-AdaSGA}
    Let the function $f$ be locally strongly convex. If the sequence $\seq{x^k}$ is generated by Algorithm~\ref{alg:AdaSGDA} with parameter
    $\tau>1$, one has
    \begin{equation*}\label{eq:elHatLinConv}
        \wh\eL(x^{k+1},x^k)\leq \left(1-\wh{q}\right)  \wh\eL(x^k,x^{k-1}),\quad\quad \forall k\in \N,
    \end{equation*}
    for the Lyapunov function $\func{\wh\eL}{\R^n\times\R^n}{\R}$ given by
    \begin{equation}\label{eq:elHat}
        \wh\eL(x^{k+1},x^k):=\|x^{k+1}-x^*\|^2+\left(\tfrac{\omega^2}{1-\omega^2}+\tfrac{\mu_\W}{2\omega L_\W }\right) \|x^{k+1}-x^k\|^2+2\alpha_k \gamma_k \Big(1+\tfrac{\theta_k}{2(1-\omega^2)}\Big) \big(f(x^k)-f^*\big),
    \end{equation}
    where $x^*\in \X^*$, $\W=\overline{B}(0;R)$ with $R$ came from Remark \ref{rem:adaSG}, and
    \begin{equation*}\label{eq:qLinear}
        \wh{q}:=\min\set{\tfrac{\mu_\W}{2}\min\left\{\alpha_0\gamma_0,\tfrac{\omega}{L_\W}\right\}, \tfrac{\mu_\W(1-\omega^2)}{2\omega^3 L_\W+\mu_\W(1-\omega^2)},\tfrac{(\tau-1)\min\left\{\tfrac{\alpha_0\gamma_0\mu_\W}{\omega},\tfrac{\mu_\W}{L_\W}\right\}}{\tau \left(2(1-\omega^2)+\min\left\{\tfrac{\alpha_0\gamma_0\mu_\W}{\omega},\tfrac{\mu_\W}{L_\W}\right\}\right)}}\in (0,1).
    \end{equation*}
\end{thm}
\begin{proof}
    Since $f$ is a locally strongly convex function and $\X^*\neq \emptyset$, it follows from Proposition~\ref{pro:SingleStrong} that $\X^*=\{x^*\}$ is singleton. Moreover, by Theorem \ref{thm:Glo-Conv-AdaSGA}, $\seq{x^k}\subseteq \W$ and $x^k\to x^*$, which ensures that $x^*\in\W$.
    Furthermore, in light of Theorem~\ref{thm:localglobal-strong} and Remark \ref{rem:adaSG}, the function $f$ is a $\mu_\W$-strongly convex and $L_\W$-smooth function on $\W$. 
    Now, we first show $\alpha_k\gamma_k$ is bounded above by $\tfrac{\omega}{\mu_{\W}}$.
    Form $\mu_\W$-strong convexity of $f$, it holds that
    \begin{align}\label{eq:stronginequality}
        \innprod{\nabla f(x^k)}{x^*-x^k}&\leq f(x^*)-f(x^k) -\tfrac{\mu_\W}{2} \|x^k-x^*\|^2,
    \end{align}
    and by Fact \ref{fac:StrongChar}~(c), we obtain
    \begin{align*}
        \|\nabla f(x^k)-\nabla f(x^{k-1})\| \|x^k-x^{k-1}\|\geq \innprod{\nabla f(x^k)-\nabla f(x^{k-1})}{x^k-x^{k-1}}\geq \mu_\W \|x^k-x^{k-1}\|^2,
    \end{align*}
    leading to the following upper bound on $\alpha_k\gamma_k$:
    \begin{equation}\label{eq:upperAlphaGamma}
        \alpha_k\gamma_k\leq \tfrac{\omega}{L_k}=\tfrac{\omega\|x^k-x^{k-1}\|}{\|\nabla f(x^k)-\nabla f(x^{k-1})\|} \leq \tfrac{\omega}{\mu_\W}.
    \end{equation}
    Next, we establish the linear convergence inequality for Lyapunov function $\wh\eL$.
    Let us set $x^0_*:=\dist(x^0;\X^*)$. By Theorem~\ref{thm:Glo-Conv-AdaSGA}, $x^*\in \interior \W$ and $\seq{x^k} \subseteq \interior\W$. Moreover, applying \eqref{eq:Upp-Bou-xk} and \eqref{eq:Uperxkx*0}
    $$\left\|x^*-\tfrac{1}{L_\W}\big(\nabla f(x^*)-\nabla f(x^k)\big)\right\|\le \|x^*\| + \tfrac{1}{L_\W}\|\nabla f(x^k)\| \leq \|x^*\| + \|x^k-x_*^0\| \leq 2\sqrt{\eta} + \dist(x^0;\X^*) + \|x^0\|< R,$$
    i.e., $x^*-\tfrac{1}{L_\W}\big(\nabla f(x^*)-\nabla f(x^k)\big) \in \interior\W$. Hence, by Corollary \ref{cor:HSmoothChar-5}~\ref{cor:HSmoothChar-5}, we get
    \begin{equation}\label{eq:Ext-Conv-Ineq1}
        f(x^*)\geq f(x^k)+\innprod{\nabla f(x^k)}{x^*-x^k}+\tfrac{1}{2L_\W}\|\nabla f(x^k)\|^2.
    \end{equation}
    It follows from \eqref{eq:Ext-Conv-Ineq1} together with \eqref{eq:upperAlphaGamma} that
    \begin{align*}
        \innprod{\nabla f(x^k)}{x^*-x^k}\leq f^*-f(x^k) -\tfrac{1}{2L_\W} \|\nabla f(x^k)\|^2\leq f^*-f(x^k) -\tfrac{\mu_\W}{2\omega L_\W \alpha_k \gamma_k} \|x^{k+1}-x^k\|^2.
    \end{align*}
    Adding inequality \eqref{eq:stronginequality} to the former inequality ensures
    \begin{align*}
        \innprod{\nabla f(x^k)}{x^*-x^k}&\leq f^*-f(x^k)-\tfrac{\mu_\W}{4} \|x^k-x^*\|^2 -\tfrac{\mu_\W}{4\omega L_\W \alpha_k\gamma_k} \|x^{k+1}-x^k\|^2.
    \end{align*}
    By applying this upper bound to equality \eqref{eq:normXk1XsUpBound11100} with $x=x^*$, we obtain
    \begin{equation}\label{eq:xk1xsUpBoundLinConv}
        \|x^{k+1}-x^*\|^2 \leq \left(1-\tfrac{\mu_\W\alpha_k\gamma_k}{2}\right)\|x^k-x^*\|^2-2\alpha_k \gamma_k (f(x^k)-f^*)+\left(1-\tfrac{\mu_\W}{2\omega L_\W}\right)\|x^{k+1}-x^k\|^2,
    \end{equation}
    which combining it with \eqref{eq:normXk1XkUpBound111} yields that
    \begin{equation}\label{eq:xk1xsUpBound}
        \begin{split}
            &\|x^{k+1}-x^*\|^2+\left(\tfrac{\omega^2}{1-\omega^2}+\tfrac{\mu_\W}{2\omega L_\W}\right) \|x^{k+1}-x^k\|^2+2\alpha_k \gamma_k \Big(1+\tfrac{\theta_k}{2(1-\omega^2)}\Big) \big(f(x^k)-f^*\big) \\
            &\leq \left(1-\tfrac{\mu_\W\alpha_k\gamma_k}{2}\right)\|x^k-x^*\|^2+\tfrac{\omega^2}{1-\omega^2}\|x^k-x^{k-1}\|^2+\tfrac{\alpha_k\gamma_k\theta_{k}}{1-\omega^2}\big(f(x^{k-1})-f^*\big)\\
            &\leq \left(1-\tfrac{\mu_\W\alpha_k\gamma_k}{2}\right)\|x^k-x^*\|^2+\tfrac{\omega^2}{1-\omega^2}\|x^k-x^{k-1}\|^2+2\alpha_{k-1}\gamma_{k-1}\left(1+\tfrac{\theta_{k-1}}{2\tau(1-\omega^2)}\right)\big(f(x^{k-1})-f^*\big)\\
            &= c_1\|x^k-x^*\|^2+c_2\|x^k-x^{k-1}\|^2+2\alpha_{k-1}\gamma_{k-1}c_3(f(x^{k-1})-f^*),
        \end{split}
    \end{equation}
    where
    \begin{align*}
        c_1:=1-\tfrac{\mu_\W\alpha_k\gamma_k}{2}, \quad c_2:=\tfrac{\omega^2}{1-\omega^2}, \quad c_3:=1+\tfrac{\theta_{k-1}}{2\tau(1-\omega^2)}.
    \end{align*}
    The terms $c_2$ and $c_3$ can be equivalently expressed as
    \begin{align*}
        c_2&= \left(1-\tfrac{\mu_\W(1-\omega^2)}{2\omega^3 L_\W+\mu_\W(1-\omega^2)}\right)\left(\tfrac{\omega^2}{1-\omega^2}+\tfrac{\mu_\W}{2\omega L_\W}\right),
    \end{align*}
    and
    \begin{align*}
        c_3=\left(1-\tfrac{(\tau-1)\theta_{k-1}}{\tau \big(2(1-\omega^2)+\theta_{k-1}\big)}\right) \Big(1+\tfrac{\theta_{k-1}}{2(1-\omega^2)}\Big).
    \end{align*}
    In light of \eqref{eq:elHat} together with the reformulated expressions of $c_2$ and $c_3$, the inequality \eqref{eq:xk1xsUpBound} can be rewritten as
    \begin{align}\label{eq:Liapanof000}
        \wh \eL(x^{k+1},x^k)\leq \left(1-\wh{q}_k\right) \wh \eL(x^k,x^{k-1}),
    \end{align}
    at which
    $$\wh{q}_k:=\min\set{\tfrac{\mu_\W\alpha_k\gamma_k}{2}, \tfrac{\mu_\W(1-\omega^2)}{2\omega^3 L_\W+\mu_\W(1-\omega^2)},\tfrac{(\tau-1)\theta_{k-1}}{\tau \big(2(1-\omega^2)+\theta_{k-1}\big)}}.$$
    To complete the proof, it suffices to show that $\wh{q}\le \wh{q}_k$ for all $k\in \N$.
    From \eqref{eq:LowerAlphaGamma} and \eqref{eq:upperAlphaGamma}, we obtain the following bounds:
    \begin{equation}\label{eq:BoundAlpaGamaTeta}
        \alpha_k\gamma_k\in \left[\min\left\{\alpha_0\gamma_0,\tfrac{\omega}{L_\W}\right\},\tfrac{\omega}{\mu_\W}\right],\quad\quad\quad\quad\theta_k\geq \min\left\{\tfrac{\alpha_0\gamma_0\mu_\W}{\omega},\tfrac{\mu_\W}{L_\W}\right\}.
    \end{equation}
    Since the function $\phi(\theta)=\tfrac{(\tau-1)\theta}{\tau \big(2(1-\omega^2)+\theta\big)}$ is monotonically nondecreasing on $\R_+$, it follows that
    $$\tfrac{(\tau-1)\theta_{k-1}}{\tau \big(2(1-\omega^2)+\theta_{k-1}\big)} \ge \tfrac{(\tau-1)\min\left\{\tfrac{\alpha_0\gamma_0\mu_\W}{\omega},\tfrac{\mu_\W}{L_\W}\right\}}{\tau \left(2(1-\omega^2)+\min\left\{\tfrac{\alpha_0\gamma_0\mu_\W}{\omega},\tfrac{\mu_\W}{L_\W}\right\}\right)}.$$
    Hence, we conclude that
    \begin{align*}
        0<\wh{q} &= \min\set{\tfrac{\mu_\W}{2}\min\left\{\alpha_0\gamma_0,\tfrac{\omega}{L_\W}\right\}, \tfrac{\mu_\W(1-\omega^2)}{2\omega^3 L_\W+\mu_\W(1-\omega^2)},\tfrac{(\tau-1)\min\left\{\tfrac{\alpha_0\gamma_0\mu_\W}{\omega},\tfrac{\mu_\W}{L_\W}\right\}}{\tau \left(2(1-\omega^2)+\min\left\{\tfrac{\alpha_0\gamma_0\mu_\W}{\omega},\tfrac{\mu_\W}{L_\W}\right\}\right)}}\\
        &\le \min\set{\tfrac{\mu_\W\alpha_k\gamma_k}{2}, \tfrac{\mu_\W(1-\omega^2)}{2\omega^3 L_\W+\mu_\W(1-\omega^2)},\tfrac{(\tau-1)\theta_{k-1}}{\tau \big(2(1-\omega^2)+\theta_{k-1}\big)}}<1,
    \end{align*}
    which implies that the inequality \eqref{eq:Liapanof000} can be uniformly bounded by
    \begin{align*}
        \wh \eL(x^{k+1},x^k)\leq \left(1-\wh{q}\right) \wh \eL(x^k,x^{k-1}),
    \end{align*}
    adjusting the claim.\qed
\end{proof}

The following corollary as a direct consequence of Theorem~\ref{thm:linConv-Lya-STR-AdaSGA} presents the linear convergence rate of AdaSGA under local strong convexity. 

\begin{thm}[\textbf{Linear convergence of AdaSGA under local strong convexity}] \label{thm:Lin-Conv-StrCon-AdaSGA}
    Let the function $f$ be locally strongly convex. If the sequence $\seq{x^k}$ is generated by Algorithm~\ref{alg:AdaSGDA} with parameter
    $\tau>1$, the sequences $\seq{\|x^{k}-x^*\|}$, $\seq{f(x^k)-f^*}$, and $\|\nabla f(x^k)\|$ converge R-linearly $0$, i.e.,
    \begin{equation}\label{eq:Lin-xk-St1}
        \|x^{k+1}-x^*\|\leq \tilde{c}_1\left(\sqrt{1-\wh{q}}\right)^k,\quad\quad \forall k\in \N_0,
    \end{equation}
    \begin{equation}\label{eq:Lin-fk-St1}
        f(x^k)-f^*\leq \tilde{c}_2 (1-\wh{q})^k,\quad\quad \forall k\in \N_0,
    \end{equation}
    \begin{equation}\label{eq:Lin-gradk-St1}
        \|\nabla f(x^k)\|\leq \tilde{c}_3 \left(\sqrt{1-\wh{q}}\right)^k,\quad\quad \forall k\in \N_0,
    \end{equation}
    where $x^*\in \X^*$,
    \[
    \tilde{c}_1:= \sqrt{(R+\|x^0\|)^2+\eta},\quad\quad
    \tilde{c}_2:= \tfrac{\omega L^2_\W \tilde{c}_1^2}{\mu_\W \min\{\alpha_0^2\gamma_0^2 L_\W^2, \omega^2\}},\quad\quad
    \tilde{c}_3:= \tfrac{\sqrt{2\tilde{c}_1^2(1-\omega^2)\omega L_{\W}^3}}{\sqrt{2\omega^2L_{\W}+(1-\omega^2)\mu_{\W}}\min\{\alpha_0\gamma_0L_\W , \omega\}},
    \]
    and the parameters $\wh{q}\in (0,1)$ and $\eta, R$ are respectively given in Theorem~\ref{thm:linConv-Lya-STR-AdaSGA} and Remark~\ref{rem:adaSG}.
\end{thm}
\begin{proof}
    Based on Theorem~\ref{thm:linConv-Lya-STR-AdaSGA}, one has
    \begin{equation}\label{eq:Lin-LWhat}
        \wh\eL(x^{k+1},x^k)\leq \left(1-\wh{q}\right)  \wh\eL(x^k,x^{k-1})\le \ldots \le \left(1-\wh{q}\right)^k  \wh\eL(x^1,x^{0}).
    \end{equation}
    From \eqref{eq:xk1xsUpBoundLinConv}, it holds that
    \[
    \|x^{1}-x^*\|^2 \leq \left(1-\tfrac{\mu_\W\alpha_0\gamma_0}{2}\right)\|x^0-x^*\|^2-2\alpha_0 \gamma_0 (f(x^0)-f^*)+\left(1-\tfrac{\mu_\W}{2\omega L_\W}\right)\|x^{1}-x^0\|^2.
    \]
    Using this inequality and the assumption $0<\omega\le \tfrac{1}{\sqrt{2}}$, we deduce
    \begin{align*}
        \wh\eL(x^{1},x^0) &=\|x^1-x^*\|^2+\left(\tfrac{\omega^2}{1-\omega^2}+\tfrac{\mu_\W}{2\omega L_\W }\right) \|x^{1}-x^0\|^2+2\alpha_0 \gamma_0 \Big(1+\tfrac{\theta_0}{2(1-\omega^2)}\Big) \big(f(x^0)-f^*\big)\\
        &\le \left(1-\tfrac{\mu_\W\alpha_0\gamma_0}{2}\right)\|x^0-x^*\|^2+\left(1+\tfrac{\omega^2}{1-\omega^2}\right)\|x^{1}-x^0\|^2+\tfrac{2\alpha_0 \gamma_0\theta_0}{2(1-\omega^2)} (f(x^0)-f^*)\le (R+\|x^0\|)^2+\eta,
    \end{align*}
    which simplifies \eqref{eq:Lin-LWhat} to $\wh\eL(x^{k+1},x^k)\leq \tilde{c}_1^2 \left(1-\wh{q}\right)^k$.
    Applying this bound and in light of definition of $\wh\eL(x^{k+1},x^k)$ in \eqref{eq:elHat}, we obtain:
    \begin{equation}\label{eq:Lin-xk-St2}
    \|x^{k+1}-x^*\|^2\le \tilde{c}_1^2 \left(1-\wh{q}\right)^k,
    \end{equation}
    \begin{equation}\label{eq:Lin-fk-St2}
    2\alpha_k \gamma_k \Big(1+\tfrac{\theta_k}{2(1-\omega^2)}\Big) \big(f(x^k)-f^*\big) \le \tilde{c}_1^2 \left(1-\wh{q}\right)^k, 
    \end{equation}
    \begin{equation}\label{eq:Lin-gradk-St2}
    \left(\tfrac{\omega^2}{1-\omega^2}+\tfrac{\mu_\W}{2\omega L_\W }\right)\alpha_k^2\gamma_k^2 \|\nabla f(x^k)\|^2 \le \tilde{c}_1^2 \left(1-\wh{q}\right)^k.
    \end{equation}
    Inequality \eqref{eq:Lin-xk-St1} follows directly from \eqref{eq:Lin-xk-St2}. Furthermore, On account of \eqref{eq:alphak}, \eqref{eq:thetak}, and the lower bounds from \eqref{eq:BoundAlpaGamaTeta}, it is concluded that
    \[
    2\alpha_k \gamma_k \Big(1+\tfrac{\theta_k}{2(1-\omega^2)}\Big) \ge \alpha_{k+1}\gamma_{k+1}\theta_{k+1} \ge \tfrac{\mu_\W \min\{\alpha_0^2\gamma_0^2 L_\W^2, \omega^2\}}{\omega L^2_\W},
    \]
    \[
    \alpha_k^2\gamma_k^2 \ge\min\left\{\alpha_0^2\gamma_0^2,\tfrac{\omega^2}{L_\W^2}\right\}. 
    \]
    Substituting these bounds into \eqref{eq:Lin-fk-St2} and \eqref{eq:Lin-gradk-St2}, respectively, yields the desired results \eqref{eq:Lin-fk-St1} and \eqref{eq:Lin-gradk-St1}, completing the proof.\qed
\end{proof}

We next derive the complexity AdaSGA under local strong convexity.

\begin{thm}[\textbf{Complexity analysis under local strong convexity}]
    Let the function $f$ be locally strongly convex. If the sequence $\seq{x^k}$ is generated by Algorithm~\ref{alg:AdaSGDA} with parameter
    $\tau>1$, for the accuracy parameter $\varepsilon>0$, there exists $x^*\in \X^*$ such that the following assertions hold:
    \begin{enumerate}
        \item
            The number of iterations to guarantee $\|x^k-x^*\|\leq \varepsilon$, $\Nf(\varepsilon)$, satisfies where
            \[\Nx(\varepsilon)\le \MN\big(\tilde{c}_1,\sqrt{1-\wh{q}}\big);\]
        \item
            The number of iterations to guarantee $f(x^k)-f^*\leq \varepsilon$, $\Nf(\varepsilon)$, satisfies
            \[\Nf(\varepsilon)\le \MN\big(\tilde{c}_2,1-\wh{q}\big);\]
        \item
            The number of iterations to guarantee $\|\nabla f(x^k)\|\leq \varepsilon$, $\Nf(\varepsilon)$, satisfies
            \[\Nnf(\varepsilon)\le \MN\big(\tilde{c}_3,\sqrt{1-\wh{q}}\big).\]
    \end{enumerate}
    The parameters $\wh{q}\in (0,1)$  is defined in Theorem~\ref{thm:linConv-Lya-STR-AdaSGA} and $\tilde{c}_1,\tilde{c}_2, \tilde{c}_3$ come from Theorem~\ref{thm:Lin-Conv-StrCon-AdaSGA}.
\end{thm}
\begin{proof}
    Employing Theorem~\ref{thm:Lin-Conv-StrCon-AdaSGA}, the proof is straightforward.\qed
\end{proof}

\section{Conclusion} \label{sec:conclusion}

In this paper, we investigated the unconstrained minimization of smooth convex functions with locally H\"{o}lder continuous gradients. After establishing several fundamental properties and characterizations of local H\"{o}lder smoothness, we analyzed the Scaled Gradient Algorithm (SGA) and proved its global convergence and complexity results. Under local strong convexity and the KL inequality, we derived linear convergence rates and showed that SGA achieves linear convergence for any KL exponent when the gradient is locally Lipschitz. Furthermore, we proposed an adaptive variant, AdaSGA, and demonstrated its global convergence and local linear rate using a Lyapunov-based approach.

Future work includes extending these results to composite or constrained settings and exploring stochastic and distributed variants of SGA.

\addcontentsline{toc}{section}{References}
\bibliographystyle{spmpsci}      
\bibliography{Bibliography}

\end{document}